\documentclass[a4paper,11pt,reqno]{amsart}
\usepackage[left=3cm, right=3cm, top=3cm, bottom=3cm]{geometry}
\usepackage{enumerate}
\usepackage{color}
\usepackage{amssymb, amsmath}
\usepackage{amscd}
\usepackage[active]{srcltx}
\usepackage[colorlinks,linkcolor={blue},citecolor={blue},urlcolor={black}]{hyperref}

\theoremstyle{plain}
\newtheorem{theorem}{Theorem}[section]
\newtheorem{corollary}[theorem]{Corollary}
\newtheorem{lemma}[theorem]{Lemma}
\newtheorem{proposition}[theorem]{Proposition}

\newtheorem{definition}[theorem]{Definition}
\newtheorem{assumption}[theorem]{Assumption}

\newtheorem*{definition*}{Definition}

\theoremstyle{remark}
\newtheorem{remark}[theorem]{Remark}
\newtheorem{example}[theorem]{Example}

\newtheorem*{claim*}{Claim}
\newtheorem*{remark*}{Remark}
\newtheorem*{example*}{Example}
\newtheorem*{notation*}{Notation}

\numberwithin{equation}{section}

\def\N{{\mathbb N}}

\def\R{{\mathbb R}}

\newcommand{\one}{{{\bf 1}}}

\renewcommand{\a}{\alpha}

\newcommand{\eps}{\varepsilon}
\renewcommand{\phi}{\varphi}

\newcommand{\bbeta}{\boldsymbol{\eta}}

\newcommand{\dd}{\mathrm{d}}
\newcommand{\dgrad}{\bar\nabla}

\DeclareMathOperator{\argmin}{argmin}

\newcommand{\ip}[1]{\left\langle {#1}\right\rangle}
\providecommand{\abs}[1]{\left\lvert#1\right\rvert}
\providecommand{\norm}[1]{\left\Vert#1\right\Vert}

\newcommand{\ddt}{\frac{\mathrm{d}}{\mathrm{d}t}}

\newcommand{\cM}{\mathcal{M}}
\newcommand{\cH}{\mathcal{H}}
\newcommand{\cB}{\mathcal{B}}

\newcommand{\cW}{\mathcal{W}}

\newcommand{\cL}{\mathcal{L}}
\newcommand{\cU}{\mathcal{U}}
\newcommand{\cV}{\mathcal{V}}

\newcommand{\cG}{\mathcal{G}}
\newcommand{\cX}{\mathcal{X}}
\newcommand{\cE}{\mathcal{E}}
\newcommand{\cA}{\mathcal{A}}

\newcommand{\CE}{\mathcal{CRE}}

\newcommand{\cF}{\mathcal{F}}

\newcommand{\cP}{\mathcal{P}}
\newcommand{\cPpE}{\mathcal{P}_{p,E}(\R^d)}
\newcommand{\cPE}{\mathcal{P}_{2,E}(\R^d)}

\renewcommand{\tilde}{\widetilde}

\newcommand{\vs}{v_*}
\newcommand{\vp}{v'}
\newcommand{\vsp}{v'_*}
\newcommand{\fs}{f_*}
\newcommand{\fp}{f'}
\newcommand{\fsp}{f'_*}

\newcommand{\bv}{\boldsymbol{v}}

\newcommand{\bu}{\boldsymbol{u}}
\newcommand{\bx}{\boldsymbol{x}}
\newcommand{\bb}{\boldsymbol{b}}

\definecolor{matthias}{rgb}{0.0,0.8,0.3}

\title[A Gradient flow approach to the Boltzmann equation]{A gradient flow approach to the Boltzmann equation}

\author{Matthias Erbar}

\address{Fakult\"at f\"ur Mathematik,
Universit\"at Bielefeld,
Postfach 100131, 33501 Bielefeld,
Germany}

\email{erbar@math.uni-bielefeld.de
}



\begin{document}

\begin{abstract}
  We show that the spatially homogeneous Boltzmann equation evolves as
  the gradient flow of the entropy with respect to a suitable geometry
  on the space of probability measures which takes the collision
  process into account. This gradient flow structure allows to give a new proof for the convergence of Kac's random walk to the
  homogeneous Boltzmann equation, exploiting the stability of gradient
  flows.
\end{abstract}


\maketitle

\section{Introduction}\label{sec:intro}

Since the pioneering work of Otto \cite{O01} it is known that many
diffusion equations can be cast as gradient flows of entropy
functionals in the space of probability measures. The relevant
geometry is given by the $L^2$ Wasserstein distance. This approach has
been used for a variety of equations as a powerful tool in the study
of the trend to equilibrium, stability questions and construction of
solutions. In each case -- as a direct consequence of the gradient
flow structure -- the driving entropy functional is non-increasing
along the solution. One of the most emblematic dissipative evolution
equations is the Boltzmann equation modelling the evolution of a dilute
gas under elastic collisions of the particles and Boltzmann's famous
H-theorem asserts that the entropy is non-increasing along its
solutions. However, uncovering a gradient flow structure for this
equation has been an open problem since \cite{O01}.

In this article we provide a solution and give a characterization of
the spatially homogeneous Boltzmann equation as a gradient flow of the
entropy. The crucial new insight is the identification of a novel
geometry on the space of probability measures that takes the collision
process between particles into account. Our main motivation to
consider this gradient structure stems from the Kac program, in
particular the propagation of chaos for Kac's stochastic many particle
systems and its convergence to the homogeneous Boltzmann equation. We provide a new proof of this result by exhibiting a gradient flow structure also for the Kac system and showing that it $\Gamma$-converges to our gradient structure for the Boltzmann equation in the spirit of Sandier--Serfaty \cite{SS04}. 

\subsection{Homogeneous Boltzmann equation and gradient flow structure}
\label{sec:gf-intro}

We consider the spatially homogeneous Boltzmann equation
\begin{align}\label{eq:boltz}
 \partial_t f = Q(f)\;,
\end{align}
where $f:\R^d\to \R_+$ is a probability density and $Q$ denotes the Boltzmann
collision operator given by
\begin{align}\label{eq:Collision-operator-alt}
Q(f)= \int_{\R^d}\int_{S^{d-1}}\big[\fp\fsp-f\fs\big]B(v-\vs,\omega)\dd \vs\dd \omega\;.   
\end{align}
Here $B$ is the collision kernel and $v,\vs$ and $\vp,\vsp$ denote the
pre- and post-collisional velocities respectively which are related
according to
\begin{align}\label{eq:pre-post2}
  \vp = v - \ip{v-\vs,\omega}\omega\;,\quad \vsp = \vs +
  \ip{v-\vs,\omega}\omega\;,\quad \omega\in S^{d-1}\;,
\end{align}
and we will often use the notation
$f=f(v)$, $\fs=f(\vs)$, $\fp=f(\vp)$, $\fsp=f(\vsp)$.
We consider regularized collision kernels with cutoff, more precisely, we
assume that $B(v-\vs,\omega)$ is bounded away from zero and comparable to $(1+|v-\vs|^2)^{\gamma/2}$ for some $\gamma\in (-\infty,1]$, see Assumption \ref{ass:kernel-bounded} for more details.

Boltzmann's H-theorem asserts that the entropy $\cH(f)=\int f\log f$
is non increasing along solutions to the Boltzmann equation, more
precisely, we have
 $ \frac{\dd}{\dd t} \cH(f_t) = - D(f_t)\leq 0\;,$
where 
\begin{align}\label{eq:intro-dissipation}
  D(f_t) = \frac14\int
  \log\frac{\fp\fsp}{f\fs}(\fp\fsp-f\fs)B(v-\vs,\omega)\dd\omega\dd
  \vs\dd v\;.
\end{align}

Let us now give a heuristic description of the gradient flow structure
of the Boltzmann equation. We recall that the gradient flow of a function
$E$ on a Riemannian manifold $M$ is given as
$\dot x_t = -\nabla E(x_t) = -K_{x_t}DE(x_t)$ with $DE$ being the
differential of $E$ and $K_x:T^*_xM\to T_xM$ the canonical map from
the cotangent to the tangent space induced by the Riemannian metric.

For the Boltzmann equation we formally take the manifold to be the set
$\cP(\R^d)$ of probability densities on $\R^d$ and the driving
functional to be the entropy $\cH$. Its differential $D\cH(f)$ at $f$
is given as $\log f=\frac{\delta\cH}{\delta f}$ in the sense that for
any tangent vector, i.e.~a function $s$ with $\int s(v)\dd v=0$, we
have
$\lim_{\varepsilon\to0}\varepsilon^{-1}\big[\cH(f+\varepsilon
s)-\cH(f)\big] =D\cH(f)[s]=\int \log f(v)s(v)\dd v$.
Identifying the gradient flow structure of the Boltzmann equation
requires to identify the right geometry on the set $\cP(\R^d)$ given
in terms of a suitable map $K$. This is achieved by introducing the
\emph{Onsager operator} $\mathcal K^B_f$ given by
\begin{align}\label{eq:Onsager}
 \mathcal K^B_f\phi (v) = -\int \bar\nabla \phi \Lambda(f) B(v-\vs,\omega)\dd\vs\dd\omega\;.
\end{align}
Here we have set $\dgrad\phi = \phi'+\phi'_*-\phi-\phi_*$ and
$\Lambda(f)$ is shorthand for $\Lambda\big(f \fs,\fp\fsp\big)$, where
$\Lambda(s,t)=(s-t)/(\log s - \log t)$ denotes the logarithmic mean.
Now the Boltzmann equation can be written as
\begin{align*}
  \partial_t f = Q(f) = -\mathcal K^B_f D\cH(f)\;,
\end{align*}
giving the desired gradient flow structure.
\smallskip

This gradient flow interpretation of the Boltzmann equation can also
be expressed by the following \emph{variational
  characterization}. Denoting by $\ip{\cdot,\cdot}_f $ the Riemannian
metric at $f$ we have for any curve $(f_t)$ of probability densities
that
\begin{align}\label{eq:intro-CS}
  \cH(f_T)- \cH(f_0) = \int_0^T\ip{\nabla \cH(f_t),\partial_tf}_{f_t}\dd t \geq -\frac12\int_0^T|\nabla \cH(f_t)|^2_{f_t} +|\partial_tf|^2_{f_t}\dd t\;.
\end{align}
Moreover, equality holds if and only if $\partial_tf=-\nabla\cH(f_t)$,
i.e.~$(f_t)$ is the gradient flow of the entropy, hence the solution
to the Boltzmann equation. In this sense, the Boltzmann equation is a
steepest descent flow decreasing the entropy \emph{as fast as
  possible}. 
\smallskip

Our first main result is a rigorous implementation of this variational
characterization. To this end we replace the formal norm of the
gradient and the speed of the curve with suitable notions. Note that
\begin{align*}
|s|^2_{f}=\int\phi\mathcal K^B_f\phi= \frac14\int |\bar\nabla\phi|^2\Lambda(f)B(v-\vs,\omega)\dd\omega\dd\vs\dd v\;,
\end{align*}
with $\phi$ such that $\mathcal K^B_f\phi=s$ and where we have
symmetrized over $v,\vs,\vp,\vsp$. In particular, the dissipation
\eqref{eq:intro-dissipation} takes the role of norm of the gradient,
i.e.~$|\nabla \cH(f)|_f^2= \int\log f\mathcal K^B_f\log f= D(f)$.

In order to define the notion of speed of a curve $(f_t)_t$, we first consider  the equation
\begin{align}\label{eq:intro-cre-psi}
  \partial_t f(v) = \mathcal K^B_{f_t}\psi_t (v)= -\int\bar \nabla\psi_t\Lambda(f)B(v-\vs,\omega)\dd \vs\dd\omega\;.
\end{align}
We perform a change of variables, setting
$U_t(v,\vs,\omega)=\bar\nabla\psi_t\Lambda(f)B(v-\vs,\omega)$ so that
\eqref{eq:intro-cre-psi} becomes linear in $(f,U)$ and reads for all
test functions $\phi$ as:
\begin{align}\label{eq:CRE-pre}
  \ddt \int \phi f_t = \frac14\int \bar\nabla\phi U_t\;.
\end{align}
This will be called \emph{collision rate equation} since $U$ governs
the evolution of the density $f$ by prescribing the rate at which
collisions happen between the particles. Now, the quantity $\int_0^T |\partial_tf|_f^2\dd t$ will be replaced by the action
\begin{align}\label{eq:action-pre}
 \cA_T(f):=\inf\left\{\frac14\int_0^T \int \frac{|U_t|^2}{\Lambda(f_t)B}\;\dd t\right\}\;,
\end{align}
where the infimum is over all $(U_t)_t$ satisfying the collision rate
equation \eqref{eq:CRE-pre}. See Section \ref{sec:CRE-action} for the
precise construction where we study \eqref{eq:CRE-pre} and \eqref{eq:action-pre} in a natural measure valued setting. Under the Assumption \ref{ass:kernel-bounded} on $B$ we then have the following variational characterization, see Theorem
\ref{thm:EDI} below.

\begin{theorem}\label{thm:EDI-intro}
   For any
  curve $(f_t)_{t\in[0,T]}$ of probability densities with
  $\cH(f_0)<\infty$ and bounded moment of order $2+\max(0,\gamma)$ we have
 \begin{align*}
   J_T(f):=\cH(f_T)-\cH(f_0)+\frac12\int_0^TD(f_t)\dd t+\frac12\cA_T(f) \geq 0.
 \end{align*}
 Moreover, we have $J_T(f)=0$ if and only if $(f_t)_{t}$ is a
 solution to the homogeneous Boltzmann equation starting from $f_0$.
\end{theorem}

We remark that this result can be recast in the framework of gradient
flows in metric spaces as developed in \cite{AGS08}. In particular it
is possible to construct the Riemannian distance $\cW_B$ on
$\cP(\R^d)$ associated with the Onsager operator $\mathcal K^B$. We
explore this point of view in the appendix.\medskip

We will also discuss a generalization of the previous theorem giving variational characterizations of the Boltzmann equation in terms of so-called \emph{generalized gradient structures}. To this end one considers a pair of primal and dual dissipation potentials $\mathcal R(f,\phi)$ and $\mathcal R^{*}(f,\xi)$ that are convex conjugated in the second variable. Then similar as in \eqref{eq:intro-CS} we have formally for any curve $(f_{t})$ of densities that 
\begin{equation*}
\mathcal H(f_{T})-\mathcal H(f_{0}) =\int_{0}^{T}\ip{D\mathcal H(f_{t}),\partial_{t}f}\dd t \geq -\int_{0}^{T}\mathcal R(f_{t},\partial_{t}f) +\mathcal R^{*}(f_{t},-D\mathcal H(f_{t}))\dd t\;.
\end{equation*}
Equality is attained if and only if 
\begin{equation}\label{eq:GF-gen}
\partial_{t}f = D_{\xi}\mathcal R^{*}(f, -D\mathcal H(f))\;.
\end{equation}
Hence the latter evolution is characterized as minimizer of the functional
\begin{equation}\label{eq:L}
\mathcal L_{T}(f):= \mathcal H(f_{T})-\mathcal H(f_{0})+\int_{0}^{T}\mathcal R(f_{t},\partial_{t}f) +\mathcal R^{*}(f_{t},-D\mathcal H(f_{t}))\dd t\;.
\end{equation}
Under suitable compatibility assumptions on $\mathcal R$ and $\mathcal H$, the resulting evolution \eqref{eq:GF-gen} is indeed the Boltzmann equation. One choice for $\mathcal R(f,\partial_{t}f)$ and $\mathcal R^{*}(f,-D\mathcal H(f))$ are the quadratic expressions $\frac12|\partial_{t}f|^{2}_{f}=\frac12\ip{\partial_{t}f,\mathcal K^{B}_{f}\partial_{t}f}$ and $\frac12|\nabla \mathcal H(f)|_{f}^{2} = \frac12\ip{D\mathcal H(f),\mathcal K_{f}^{B}D\mathcal H(f)}$ by which we recover the gradient flow structure already discussed.
One compelling motivation for such generalized gradient structures comes from the fact that in many situations they arise naturally from the analysis of large deviations for an underlying microscopic particle system whose limiting behavior is described by \eqref{eq:GF-gen}. Namely, the functional $\mathcal L_{T}$ appears as the rate function for large deviations on the path level, see e.g.~\cite{MPR14} for an in depth discussion. In the construction of the generalized gradient structure, we follow the approach of \cite{PRST20}, where such structures have been analyzed in detail in the context of jump processes.

In the present setting we obtain the following result. Fix a pair of even, lower semi-continuous convex conjugated functions $\Psi, \Psi^{*}:\R\to[0,\infty)$ with $\Psi(0)=\Psi^{*}(0)=0$ and a 1-homogeneous concave function $\theta:[0,\infty)\times[0,\infty)\to [0,\infty)$ such that the compatibility condition
\[(\Psi^{*})'(\log s-\log t)\theta(s,t) = s-t\qquad\forall s,t>0\]
holds (see Assumption \ref{ass:gen-grad} for additional assumptions on $\Psi^{*}$ and $\theta$). Set
\begin{align*}
\mathcal R(f,U)&:=\frac14\int \Psi\Big(\frac{U}{\theta(f)B}\Big)\theta(f)B\dd v\dd\vs\dd\omega\;,\\
\mathcal D_{\Psi^{*}}(f)&:=\mathcal R^{*}(f, -DH(f)):=\frac14\int \Psi^{*}(-\bar\nabla \log f)\theta(f)B\dd v\dd\vs\dd\omega\;,
\end{align*}
where we have set $\theta(f):=\theta(f\fs,\fp\fsp)$. Then we have (see Thm.~\ref{thm:ED-balance} below):
\begin{theorem}\label{thm:intro-ED-balance}
  For any curve of probability densities $(f_{t})$ with $\cH(f_0)<\infty$ and bounded moment of order $2+\max(0,\gamma)$ and $(U_{t})$ such that the collision rate equation \eqref{eq:CRE-pre} holds we have that:
 \begin{align}\label{eq:intro-alt-chain-rule-est}
   \mathcal L_T(f,U):=\cH(f_T)-\cH(f_0)+\int_0^T D_{\Psi^{*}}(f_t) + \mathcal R(f_{t},U_{t})\dd t \geq 0 \;.
 \end{align}
 Moreover, we have $\mathcal L_T(f,U)=0$ if and only if $(f_{t})$ is a solution to
 the homogeneous Boltzmann equation and $U_{t}=(f\fs-\fp\fsp)B$.
\end{theorem}

This generalized gradient structure encompasses in particular the previous quadratic structure by choosing $\theta=\Lambda$ the logarithmic mean and $\Psi(\xi)=\Psi^{*}(\xi)=\frac12 |\xi|^{2}$. Another particular choice of interest is
\[\theta(s,t)=\sqrt{st}\;,\qquad \Psi^{*}(\xi)=4\big(\cosh(\xi/2)-1\big)\;.\]
This particular variational structure seems to have been explicitly identified for the first time by Grmela, see e.g.~\cite[Eq. (A7)]{Grm93}.
For jump processes a similar structure is connected with the large deviations on the path level for the empirical measure of a growing number of independent particles, see e.g.~\cite{PRST20}. Here, the resulting structure can at least formally can be related to the large deviations of the Kac particle system that we describe below.

\subsection{Consistency for Kac's random walk}
\label{sec:kac-intro}

A central motivation for considering the gradient flow structure just
described is to give a new proof of the convergence of Kac's random
walk to the solution of the spatially homogeneous Boltzmann
equation. Kac introduced his random walk in the seminal work
\cite{Kac56} as a probabilistic model for $N$ colliding particles. It
is a continuous time Markov chain on the set $\cX_N$ of $N$ velocities
with fixed momentum and energy,
\begin{align*}
  \cX_N := \left\{(v_1,\dots,v_N)\in\R^{dN} ~|~\sum_{i=1}^Nv_i=0\;,~\sum_{i=1}^N|v_i|^2=Nd  \right\}\;.
\end{align*}
In each step, two uniformly chosen particles $i,j$ collide, i.e.~$\bv$
is updated to
$R^\omega_{ij}\bv = (v_1,\dots,v_i',\dots,v_j',\dots,v_N)$ where
$v_i' = v_i - \ip{v_i-v_j,\omega}\omega$ and
$v_j' = v_j + \ip{v_i-v_j,\omega}\omega$ with a random collision parameter $\omega\in S^{d-1}$ distributed according to $B(v_i-v_j,\cdot)$. The rate is chosen such that
on average $N$ collisions occur per unit of time. More precisely, the
generator of the Markov chain is given by
\begin{align}\label{eq:Kac-generator-intro}
  A f(\bv) = \frac{1}{2N}\int_{S^{d-1}}\sum_{i,j}\left[f(R^\omega_{ij}\bv)-f(\bv)\right]B(v_i-v_j,\omega)\dd \omega\;.
\end{align}
The Markov chain is reversible with respect to the Hausdorff measure
$\pi_N$ on $\cX_N$. If we denote by $\mu^N_t$ the law of the Markov
chain starting from $\mu^N_0$, then its density $f^N_t$ w.r.t.~$\pi_N$
satisfies Kac's master equation $\partial_t f^N_t=A f^N_t$.

A natural way to study the convergence of Kac's random walk to the
Boltzmann equation is via its empirical measures
$L_N(\bv)=\frac1N\sum_{i=1}^N\delta_{v_i}\in\cP(\R^d)$. We will show
the following:

\begin{theorem}\label{thm:kac-boltz-intro}
  Let $B$ satisfy Assumption \ref{ass:kernel-bounded}. For each $N$
  let $(\mu^N_t)_{t\geq0}$ be the law of Kac's random walk starting
  from $\mu^N_0$ and denote by
  $c^N_t:=(L_N)_\#\mu^N_t\in\cP(\cP(\R^d))$ the law of its
  empirical measures. Assume that $\mu^N_0$ is well-prepared for some
  $\nu_0=f_0\cL\in\cP(\R^d)$ with $\cH(\nu_0)<\infty$ (if $\gamma>0$ assume in addition finite 4th moment of $\nu_0$) in the sense
  that in the limit $N\to\infty$
  \begin{align*}
    c^N_0\rightharpoonup \delta_{\nu_0}\;,\quad
    \frac1N\cH(\mu^N_0|\pi_N) \rightarrow \cH(\nu_0|M)\;.
  \end{align*}
  Assume further that for some $p>2+\max(0,\gamma)$  
  \begin{align*}
    \sup_N \ip{\cE^N_p,\mu^N_0} <\infty\;, \qquad \cE^N_p(\bv):=\frac1N\sum_{i=1}^N|v_i|^p\;.
  \end{align*}
  Then, for all $t>0$, as $N\to\infty$ we have
   \begin{align}\label{eq:prop-chaos-intro}
    c^N_t\rightharpoonup \delta_{\nu_t}\;,\quad \frac1N\cH(\mu^N_t|\pi_N) \rightarrow \cH(\nu_t|M)\;,
  \end{align}
 where $\nu_t=f_t\cL$ and $f_t$ is the unique
  solution to the spatially homogeneous Boltzmann equation with
  initial datum $f_0$.
\end{theorem}
Here $\cH(\cdot|\pi_N)$ denotes the relative entropy w.r.t~$\pi_N$ and
$\cH(\cdot|M)$ the relative entropy w.r.t.~the standard Gaussian
density $M$ in $\R^d$. Note that the well-preparedness assumption is
satisfied for instance if the initial velocities are independent,
i.e.~$\mu_0^N=\nu_0^{\otimes N}$.  An important feature of Kac's model
is the \emph{propagation of chaos}: if the initial distribution of
velocities is asymptotically independent as $N\to\infty$ then the same
holds for all times. One way of making this precise is the convergence
\eqref{eq:prop-chaos-intro}, which is usually called entropic
propagation of chaos. This is motivated by the fact that for a true
product measure we have $\cH(\nu^{\otimes N})=N\cdot\cH(\nu)$.

We point out that the previous theorem is well-known even for a larger
class of collision kernels, see the references below. The contribution
we make here is to provide a new angle of attack on this problem by
exploiting the gradient flow structure. We will use the stability of
gradient flows following the approach of Sandier--Serfaty
\cite{SS04}. It turns out that Kac's random walk is the gradient flow
of the entropy $\cH(\cdot|\pi_N)$ in $\cP(\cX_N)$ equipped with a suitable
geometry, as we shall make precise in Section \ref{sec:kac-gf}. In
particular, the energy dissipation identity
\begin{align}\label{eq:EDI-Kac-intro}
  J^N_T(\mu^N)=\cH(\mu^N_t|\pi_N)-\cH(\mu^N_0|\pi_N)+\frac12\int_0^TD^N(\mu^N_t)\dd t+\frac12 \cA^N_T(\mu^N)
  = 0
\end{align}
holds, where $D^N$ is the dissipation of $\cH(\cdot|\pi_N)$ along the master
equation and $\cA^N_T(\mu^N)$ is the action. This is based on results
for general Markov chains and jump processes in
\cite{Ma11,Mie11,Erb14}. To obtain the desired convergence to the
Boltzmann equation it is sufficient together with some compactness to
prove convergence (in fact only $\liminf$ estimates) for the
constituent elements of the gradient flow structure, the entropy,
dissipation and the action, which allow to pass to the limit in
\eqref{eq:EDI-Kac-intro}.

\subsection{Connection to the literature}
\label{sec:lit}

For an overview of results for the spatially homogeneous Boltzmann
equation, we refer to the review by Desvillettes
\cite{DES01}. Modifications of the Wasserstein geometry have been
studied in works by Maas \cite{Ma11} and Mielke \cite{Mie11}
where gradient flow structures for finite Markov chains and
reaction-diffusion equations have been found. The gradient flow
structure for the homogeneous Boltzmann equation obtained here is
related to the discrete framework of reaction equations in
\cite{Mie11}. Formally, the homogeneous Boltzmann equation could be
seen as a binary reaction equation with a continuum of species indexed
by the velocity. Recently, a gradient flow characterization for the homogeneous Landau equation has been 
given in \cite{CDDW20} using a similar approach as in this work. Spatially inhomogeneous linear Boltzmann equations have been characterized variationally in \cite{BBB17} using a variant of the energy dissipation identity. The underlying structure is non-quadratic and inspired by large deviations. We also mention the work \cite{BBBO21}, where the large deviations of a Kac type particle system with conservation of momentum but not of energy have been determined and a corresponding generalized gradient structure for the limiting Boltzmann type equation has been established.
  
Theorem \ref{thm:kac-boltz-intro} on the convergence of Kac's random
walk goes back to Kac who proved an analogue for a simplified model
with one-dimensional velocities in \cite{Kac56}. The first proof of
convergence to the homogeneous Boltzmann equation for the model
considered here is due to Sznitman \cite{Szn84}. In both cases more
general collision kernels than in this article are considered, including in
particular the case of hard spheres. Quantitative convergence
results in Wasserstein distance were obtained later by
Mischler--Mouhot \cite{MM13}, Norris \cite{Nor16}, and Cortez--Fontbona \cite{CF18}.  Quantitative estimates for the entropic propagation of chaos in the Kac model, i.e.~on the speed of convergence in \eqref{eq:prop-chaos-intro} have been given by Carrapatoso \cite{Car15}. Similar results have been obtained for the Landau equation and other related models, see e.g.~\cite{FH16, FG17, JW18}.

\subsection{Organization}
In Section \ref{sec:prelim} we collect necessary preliminaries. In
Section \ref{sec:CRE-action} we introduce the collision rate equation
and the action of a curve. The characterization of the Boltzmann
equation as entropic gradient flow is obtained in Section
\ref{sec:gradflow}. In Section \ref{sec:kac} we exhibit a gradient
flow structure for Kac's random walk and prove its convergence to the
Boltzmann equation.

The appendices \ref{sec:metric}, \ref{sec:metricgf}, and
\ref{sec:minmov} contain the construction of the distance associated
to the Onsager operator, a reformulation of our results in the
framework of gradient flows in metric spaces, and a variational
approximation scheme for the Boltzmann equation based on the gradient
structure.

\subsection*{Acknowledgment} {\small The author is grateful for the
  remarks and suggestions of the referees on a previous version of the
  manuscript. He wishes to thank Andr\'e Schlichting for fruitful
  discussions on this work and related topics. The author gratefully
  acknowledges support by the German Research Foundation through the
  Collaborative Research Center 1060 \emph{The Mathematics of Emergent
    Effects} and the Hausdorff Center for Mathematics}

\section{Preliminaries}\label{sec:prelim}

\subsection{Homogeneous Boltzmann equation, entropy and dissipation}
\label{sec:boltz-recollect}

Let $d\geq 3$. We denote by $\cP(\R^d)$ the space of Borel probability
measures on $\R^d$ equipped with the topology of weak convergence in
duality with bounded continuous functions. We denote by $\cH(\mu)$ the
\emph{Boltzmann--Shannon entropy} defined for $\mu\in\cP(\R^d)$ by
\begin{align*}
  \cH(\mu) = \int f(v)\log f(v) \dd v\;,
\end{align*}
provided $\mu=f\cL$ is absolutely continuous w.r.t.~Lebesgue measure
$\cL$ and $\max(f\log f,0)$ is integrable, otherwise we set
$\cH(\mu)=+\infty$. We will also write $\cH(f)$ if $\mu=f\cL$.

For $p\geq1$, let $\cP_p(\R^d)=\{\mu\in\cP(\R^d)~:~\int|v|^p\dd\mu(v)<\infty\}$
denote the set of probability measures with finite moment of order $p$.
We will write \begin{equation}
  \label{eq:pmoment}
  \mathcal E_p(\mu):=\int|v|^p\dd\mu(v)\;.
\end{equation}
For $\mu\in\cP_2(\R^d)$ we define by
\begin{align}\label{eq:momentum-energy}
  \cM(\mu):=\int v\dd\mu(v)\;,\quad \mathcal E(\mu):=\mathcal E_2(\mu)=\int|v|^2\dd\mu(v)\;,
\end{align}
the momentum and energy of $\mu$. For $E>0$ we let
\begin{align}\label{eq:P2E}
  \cPpE:=\{\mu\in\cP_p(\R^d)~:~\cE_p(\mu)\leq E\}\;,
\end{align}
the set of measures with energy less than $E$. Note that $\cPpE$ is
compact for the weak topology. For $m\in\R^d$ and $E>0$ we let 
\begin{align*}
M^{m,E}(v)=\frac{1}{(2\pi E)^{d/2}}
\exp\left(-\frac{|v-m|^2}{2E}\right)\;,
\end{align*}
denote the Maxwellian or Gaussian density distribution with momentum
$m$ and energy $Ed$. The relative entropy w.r.t.~$M^{m,E}$ of a probability measure $\mu=f \cL$ is defined by 
\begin{align}\label{eq:ent-gauss}
  \cH(\mu|M^{m,E}) = \int f(v) \log \frac{f(v)}{M^{m,E}(v)}\dd v\;. 
\end{align}
For any $\mu\in \cP_2(\R^d)$ we have
\begin{align}\label{eq:ent-lowerbound}
  \cH(\mu) = \cH(\mu|M^{m,E}) - \frac12\int\frac1{E}|v-m|^2\mu(\dd v) -\frac{d}{2}\log(2\pi E)\;.
\end{align}
By Jensen's inequality we have $\cH(\cdot|M^{m,E})\geq 0$. Hence, we
see that $\cH$ is bounded below on $\cPE$. Moreover, we have that
$\cH(\mu)=\cH(\mu|M^{\mu})+\cH(M^\mu)$.  Finally, we note that $\cH$
is lower semicontinuous on $\cP_2(\R^d)$ w.r.t.~weak convergence. This
follows from the corresponding property of $\cH(\cdot|M^{m,E})$ and
lower semicontinuity of moments.  \smallskip

We collect some well known results on existence and uniqueness and
propagation of integrability for the homogeneous Boltzmann
equation.  We use the notation
\begin{equation*}
  \ip{k}:=\sqrt{1+|k|^2}\;,\quad k\in \R^d\;.
\end{equation*}
Moreover, we use the weigthed $L^1$ spaces $L^1_s(\R^d):=\{f\in L^1(\R^d): \int \ip{v}^s|f(v)|\dd v<\infty\}$.

Throughout this article we make the following assumption on
the collision kernel.

\begin{assumption}\label{ass:kernel-bounded}
  $B:\R^d\times S^{d-1}\to\R_+$ is measurable, continuous w.r.t.~the first variable, invariant under the transformation \eqref{eq:pre-post2}, and
   there exist constants
  $\gamma\in (-\infty, 1]$ and $c_B>0$ such that for all $k\in\R^d,\omega\in S^{d-1}$:
  \begin{align}\label{eq:kernel-bounded}
    c_B^{-1}\ip{k}^\gamma \leq B(k,\omega) \leq c_B\ip{k}^\gamma\;.
  \end{align}
\end{assumption}

Let us recall that typical choices of the collision kernel motivated
on physical grounds take the form $B(k,\omega)=|k|^\gamma b(\alpha)$
with $\alpha$ the angle between $k$ and $\omega$. The assumption above
corresponds to an angular cutt-off assumption removing the typical
singularity in $b$ as well as a regularization near $k=0$ removing the
singularity for $\gamma<0$ and ensuring boundedness away from zero for
$\gamma>0$.

  \begin{theorem}\label{thm:boltzmann}
 Let $f_0:\R^d\to\R_+$ be such that
 \begin{align*}
   \int_{\R^d} \ip{v}^2f_0(v)\dd v <\infty\;,\quad \int f_0(v)\log f_0(v)\dd v < \infty\;.
 \end{align*}
 If $\gamma>0$ assume in addition that $f_0\in L^1_4(\R^d)$. Then there
 exists a unique non-negative solution
 $f\in C\big([0,\infty);L^1(\R^d)\big)\cap
 L^\infty\big((0,\infty);L^1_2(\R^d)\big)$ to the homogeneous
 Boltzmann equation \eqref{eq:boltz} conserving mass, momentum and
 energy, i.e.
 \begin{align*}
   \int \big(1,v,|v|^2\big)f_t(v)\dd v =    \int \big(1,v,|v|^2\big)f_0(v)\dd v\quad\forall t\geq 0\;.
 \end{align*}
Moreover, we have for all $t>0$:
 \begin{align}\label{eq:energy-identity}
  \cH(f_t)-\cH(f_s) &= - \int_0^tD(f_r)\dd r\leq 0\;,
\end{align}
where
\begin{align}\label{eq:def-dissipation}
  D(f):= \int_{\R^{2N}}\int_{S^{d-1}}\log\frac{\fp\fsp}{f\fs}\big[\fp\fsp-f\fs\big]B(v-\vs,\omega)\dd v\dd\vs\dd\omega\;.
\end{align}
\end{theorem}
\begin{proof}
  For existence, conservation of mass, momentum, and energy, as well
  as uniqueness we refer to \cite{Ark72}. The entropy identity
  \eqref{eq:energy-identity} is proven in \cite{Lu99}.
\end{proof}

We note that for conventional hard potential kernels of the form
$B(k,\omega)=|k|^\gamma b(\theta)$, $\gamma\in (0,1]$ uniqueness of
conservative solutions is known assuming only finite energy of the
initial datum \cite{MW99}. For general kernels as considered here we
could not retrieve such an improved result in the literature.

The quantity $D(f)$ is called the \emph{entropy dissipation}. More
generally, we define the entropy dissipation $D(\mu)$ for a
probability measure $\mu$ by setting $D(\mu)=D(f)$, provided
$\mu=f\cL$ is absolutely continuous and $+\infty$ otherwise.

\subsection{Regularization by convolution}
\label{sec:OU}

For $t>0$, we consider the Maxwellian distribution 
\begin{align*}
  M_t(v) = \frac{1}{(2\pi t)^{d/2}}\exp\Big(\frac{|v|^2}{2t}\Big)\;,
\end{align*}
and note that $\int |v|^2M_t(v)\dd v=2t$. We write $M:=M_1$.

For any non-negative $f\in L^1$ with $\|f\|_{L^1}=1$, $M_t*f$ is $C^\infty$ with the bounds
\begin{align}\label{eq:Gaussian-bound}
  |M_t*f|\leq C_t\;,\quad |\log M_t*f|(v)\leq C_t(1+|v|^2)\;,
\end{align}
for a suitable constant $C_t$, see for instance \cite{CC92}. 

For fixed $\omega\in S^{d-1}$ we will denote by $T_\omega$ the
transformation $(v,\vs)\mapsto(\vp,\vsp)$ with $\vp,\vsp$ given by
\eqref{eq:pre-post2}. Note that $T_\omega$ is involutive and has unit
Jacobian determinant. We will set
\begin{align*}
  X=(v,\vs),\quad X'=(\vp,\vsp) = T_\omega X\;. 
\end{align*}
By abuse of notation we denote the Maxwellian distribution
 in $\R^{2d}$ again by $M_t$. Note that $M_t(X):=M_t(v)M_t(\vs)$. For a function
$F:\R^d\times\R^d\to\R$ we will set
\begin{align*}
  T_\omega F (X) := F(T_\omega X)\;.
\end{align*}
Convolution behaves well under tensorization. More precisely,
if for a function $f:\R^d\to\R $ we set $F=f\otimes f$,
i.e.~$F(X)=f\fs$, then we have
\begin{align*}
  F*M_t = (f*M_t)\otimes(f*M_t)\;. 
\end{align*}

The following commutation relation with the pre-post-collision change
of variables will be crucial in the sequel. It can be found in
\cite[Prop.~4] {TV99}. For the reader's convenience we will give the
short proof.

\begin{lemma}\label{lem:OU-commutation}
  Let $F:\R^{2d}\to\R$. Then, we have that for each $\omega\in
  S^{d-1}$ and any $t> 0$:
  \begin{align}\label{eq:commutation-scaling-convolution}
   (T_\omega F)*M_t = T_\omega (F*M_t) \;.
  \end{align}
  If $F=f\fs$ we have for short $M_t*(\fp\fsp)=(M_t*f)'(M_t*f)'_*$.
\end{lemma}

\begin{proof}
  First note that
  $M_t(T_\omega X)=M_t(X)$, since the relation between pre-
  and post-collisional velocities is such that
  $|v|^2+|\vs|^2=|\vp|^2+|\vsp|^2$. Using also the fact that $T_\omega$ is
  involutive with unit determinant, we find
 \begin{align*}
    &\big((T_\omega F)*M_t\big)(X) 
  = \int F(T_\omega Y)M_t(X-Y)\dd Y = \int F(Y)M_t(X-T_\omega^{-1}Y)\dd Y\\
&= \int F(Y)M_t(T_\omega X-Y)\dd Y = (F*M_t)(T_\omega X)\;.
 \end{align*}
\end{proof}

  \begin{lemma}\label{lem:kernel-conv}
    For any $p\in\R$ and $0<\delta<1$ we have
    \[\int_{\R^d}\ip{w}^pM_\delta(v-w)\dd w \leq C \ip{v}^p\;,\]
for a constant $C$ depending only on $|p|$ and $m_{|p|}(M)=\int |v|^{|p|}M(v)\dd v$.
\end{lemma}
\begin{proof}
  We use the fact that for any $p\in \R$ and $x,y\in\R^d$
  \begin{equation}
    \label{eq:petree}
    \frac{\ip{x}^p}{\ip{y}^p}\leq 2^{|p|/2}\ip{x-y}^{|p|}\;,
  \end{equation}
  known as Peetre's inequality. \eqref{eq:petree} can be readily checked for $p=2$. Taking non-negative powers yields \eqref{eq:petree} for $p\geq0$. Reversing the role of $x$ and $y$ and taking positive powers yields the statement for $p<0$.

  Now, using \eqref{eq:petree} we can estimate
  \begin{align*}
    \int \ip{w}^p M_\delta(v-w)\dd w \leq 2^{|p|/2}\ip{v}^p\int \ip{v-w}M_\delta(v-w)\dd w
    = 2^{|p|/2}\ip{v}^p \int (1+\delta^2|w|^2)^{|p|/2}M(w)\dd w\;,
  \end{align*}
  and the claim readily follows.
  \end{proof}

\subsection{Integral functionals on measures}
\label{sec:func}
We provide here basic results on integral functionals on measures that will be often used in the following.

Let $\mathcal X$ be locally compact Polish space. We denote by $\cM(\mathcal X;\R^n)$
the space of vector-valued Borel measures with finite variation on
$\mathcal X$. It will be endowed with the weak* topology of convergence in
duality with $C_0(\mathcal X;\R^n)$, i.e.~continuous functions vanishing at
infinity. Let $f:\R^n\to [0,\infty]$ be a convex, lower
semicontinuous, and superlinear and let $\sigma$ be a non-negative finite Borel measure on $\mathcal X$. Define on
$\cM(\mathcal X;\R^n)$ the functional $\mathcal F(\cdot|\sigma)$ via
  \begin{align}\label{eq:int-funct}
    \cF_{f}(\gamma | \sigma)=\int_{\mathcal X}f\left(\frac{\dd\gamma}{\dd\sigma}\right)\dd\sigma\;.
  \end{align}
  and set $\mathcal F_{f}(\gamma|\sigma)=+\infty$ if $\lambda$ is not absolutely continuous w.r.t.~$\gamma$. Note that the definition
  is independent of the choice of $\sigma$ if $f$ is positively $1$-homogeneous, i.e.~we have $f(\lambda r)=\lambda f(r)$ for all $r\in\R^{n}$ and $\lambda\geq 0$. We will write $\mathcal F_{f}(\cdot)$ instead of $\mathcal   F_{f}(\cdot|\sigma)$ in this case.
  
  \begin{lemma}\label{lem:Fprops}
  \begin{itemize}
  \item[(i)] $\cF_{f}(\cdot|\sigma)$ is convex and sequentially lower semicontinuous
    w.r.t.~weak* convergence.
  \item[(ii)] Assume that $f$ is $1$-homogeneous. If $\mathcal Y$ is another locally~compact Polish space and
    $T:\mathcal X\to \mathcal Y$ is Borel measurable, then we have that
    $\cF_{f}(T_\#\gamma) \leq \cF_{f}(\gamma)$ for all $\gamma$, where $\cF$
    is defined analogously on $\cM(\mathcal Y;\R^n)$.
\end{itemize}
\end{lemma}
\begin{proof}
  (i) This is proven in \cite[Thm.~3.4.3]{Bu89}.\\
%
(ii) Let $\bar\gamma^i=T_\#\gamma^i$ and $\bar\sigma=T_\#\sigma$.  Let $(\sigma_y)_{y\in \mathcal Y}$ be a disintegration of $\sigma$ w.r.t.~$\bar\sigma$. I.e.~$\sigma_y$ are measures on $\mathcal X$ such that $y\mapsto \sigma_y(E)$ is Borel measurable for all Borel sets $E\subset \mathcal X$, $\sigma_y(E)=\sigma_y(E\cap T^{-1}(y))$, $\sigma_y(\mathcal X)=\sigma(\mathcal X)$ for all $y$, and we have $\sigma(E)=\int \sigma_y(E)\dd \bar\sigma(y)$. Write $\lambda=\rho\sigma$, and note that we have~$\bar\lambda=\bar\rho\bar\sigma$ with $\bar\rho(y):=\int\rho(x)\sigma_y(\dd x)$. Now put $\rho_y(x)=\rho(x)/\bar\rho(y)$. Then we have 
\begin{align*}
  \cF_{f}(T_\#\gamma)&=\int_{\mathcal Y}f\big[\bar\rho\big]\dd\bar\sigma
 = \int_{\mathcal Y}f\Big[\int_{\mathcal X}\rho_y\dd\sigma_y\bar\rho(y)\Big]\bar\sigma(\dd y)\\
 &\leq \int_{\mathcal Y}\int_{\mathcal X}f\big[\rho_y(x)\bar\rho(y)\big]\sigma_y(\dd x)\bar\sigma(\dd y)
=\int f\big[\rho\big]\dd\sigma = \cF_{f}(\gamma)\;,
\end{align*}
where we have used Jensen's inequality due to the convexity and
homogeneity of $\alpha$.
 \end{proof}

As a first consequence we obtain

\begin{lemma}[Lower semicontinuity of dissipation]\label{lem:diss-lsc}
  For any sequence $(\mu_n)$ in $\cP(\R^d)$ converging weakly to $\mu$
  we have that
  \begin{align}\label{eq:lsc-dissipation}
    D(\mu) \leq \liminf_n D(\mu_n)\;.
  \end{align}
\end{lemma}

\begin{proof}
  Consider the convex, lower semicontinuous, and $1$-homogeneous
  function $G(s,t)=\frac14(t-s)(\log t-\log s)$. For $\mu\in\cP(\R^d)$ define non-negative measures $\mu^1,\mu^2\in \cM_+(\Omega)$ by
 \begin{align*}
  \mu^1(\dd v,\dd \vs,\dd \omega) := B(v-\vs,\omega)\mu(\dd v)\mu(\dd\vs)\dd\omega\;,\quad \mu^2 := T_\#\mu^1\;,
\end{align*}
where $T$ is the change of variables
$(v,\vs,\omega)\mapsto (T_\omega(v,\vs),\omega)$ between pre- and
post-collisional variables defined in \eqref{eq:pre-post2}. We note that 
  \begin{align*}
   D(\mu)=\cG(\mu^1,\mu^2):=\int G\left(\frac{\dd\mu^{1}}{\dd\sigma},\frac{\dd\mu^{2}}{\dd\sigma}\right)\dd\sigma\;,
  \end{align*}
  where $\sigma$ is any measure such that $\mu^1,\mu^2\ll\sigma$.
  Note that by the Assumption \ref{ass:kernel-bounded} on the
  collision kernel $B$, the weak convergence of $\mu_n$ to $\mu$
  implies the weak* convergence of $\mu^i_n$ to $\mu^i$ in $\cM(\Omega)$
  for $i=1,2$. Now the claim follows immediately from Lemma
  \ref{lem:Fprops}.
\end{proof}

\section{Collision rate equation and action}
\label{sec:CRE-action}

In this section, we rigorously define the notion of speed of a curve
$(f_t)_t$ associated to the formal Onsager operator $\mathcal K^B$. In
the next subsection we study the collision rate equation
\eqref{eq:CRE-pre} in a measure-valued framework replacing $f_t$ with
probability measures $\mu_t$ and $U_t$ with a family of signed
measures on $\R^d\times\R^d\times S^{d-1}$. In Subsection
\ref{sec:action} we study the action functional \eqref{eq:action-pre}
on measures and define the action of a curve.

\subsection{The collision rate equation}\label{sec:cre}
Let us set 
$$\Omega~=~\R^d\times\R^d\times S^{d-1}$$
and denote by $\cM(\Omega)$ the space of signed Borel measures with finite
variation on $\Omega$ equipped with the weak* topology in duality with
continuous functions vanishing at infinity. Recall that $\cP(\R^d)$
denotes the space of Borel probability measures on $\R^d$ equipped
with the topology of weak convergence in duality with bounded
continuous functions.

We define solutions to the collision rate equation in the following way.
\begin{definition}[Collision rate equation]\label{def:ce}
  We denote by $\CE_{T}$ the set of all pairs
  $(\mu,\cU)$ satisfying the following conditions:
  \begin{itemize}
  \item[(i)]  $\mu : [0,T] \to \cP(\R^d)$   is weakly continuous;
  \item[(ii)] $(\cU_t)_{t\in[0,T]}$ is a Borel family of measures in
    $\cM(\Omega)$;
  \item[(iii)] $\int_0^T\abs{\cU_t}(Y)\dd t~<~\infty$;
  \item[(iv)] for any $\phi\in C_b(\R^d)$ we have in the sense of distributions:
    \begin{align}\label{eq:cre-dist}
        \ddt\int\phi\dd\mu_t = \frac14\int\dgrad\phi\dd\cU_t\;.
    \end{align}
  \end{itemize}
  Moreover, we will denote by $\CE_T(\bar\mu_0,\bar\mu_1)$ the set of
  pairs $(\mu,\cU)\in\CE_T$ satisfying in addition:~
  $\mu_0=\bar\mu_0,\ \mu_1=\bar\mu_1$.
\end{definition}

Note that the integrability condition (iii) ensures that the right hand
side in (iv) is well-defined. The measures $\cU_t$ will be called
\emph{collision rates}.

\begin{remark}\label{rem:refined}
 If $(\mu,\cU)\in\CE_T$, then for any $\phi\in C_b(\R^d)$ and $0\leq t_0\leq t_1\leq T$ we have
  \begin{equation}\label{eq:cedistributionrefined}
    \int\varphi \dd\mu_{t_1}-\int\varphi \dd\mu_{t_0}~=~\frac14\int_{t_0}^{t_1}\int\dgrad\varphi \dd\cU_t\dd t\;.
  \end{equation}
  This follows readily from (iv) together with the continuity of
  $t\mapsto\mu_t$ in (i).

   The curve $(\mu_t)_{t\in[0,T]}$ is also
   absolutely continuous w.r.t.~the total variation norm. Indeed,
   from \eqref{eq:cedistributionrefined} we infer
   \begin{align*}
     \Big|\int \phi \dd(\mu_{t_1}-\mu_{t_0})\Big| \leq |\phi|_\infty\int_{t_0}^{t_1}|\cU_t|(\Omega)\dd t\;,
   \end{align*}
   and hence
   $||\mu_{t_1}-\mu_{t_0}||_{\text{TV}}\leq \int_{t_0}^{t_1}|\cU_t|\dd
   t$.
   Moreover, the distribution $\partial_t\mu_t$ on $[0,T]\times\R^d$
   is actually a signed measure with total variation bounded by
   $\int_0^T|\cU_t|(\Omega)\dd t$.
\end{remark}

\begin{remark}\label{rem:generaltest}
  The continuity equation can sometimes be tested against more
  general test functions. For instance, let $(\mu,\cU)\in\CE_T$ and
  let $\cU$ satisfy the stronger integrability condition
  \begin{align}\label{eq:more-integrability-U}
    \quad\qquad \int_0^T\int \big[\ip{v}^p+\ip{\vs}^p\big]\dd\abs{\cU_t}\dd t~<~\infty\;,
  \end{align}
  for some $p>0$. Then \eqref{eq:cedistributionrefined} holds for all $\phi:\R^d\to\R$
  continuous satisfying the growth condition~$|\phi(v)|\leq c\ip{v}^p$. This follows immediately by approximation with functions in $C_b$ and the trivial estimate $\ip{\vp}^p+\ip{\vsp}^p\leq C_p\big(\ip{v}^p+\ip{\vs}^p)$. If $\mu_t$ has density $f_t$ w.r.t.~ Lebesgue measure, we infer as above that 
  \begin{align*}
     \Big|\int \ip{v}^p\phi(v) \big(f_{t_1}(v)-f_{t_0}(v)\big)\dd v\Big| \leq C|\phi|_\infty\int_{t_0}^{t_1}\int\big[\ip{v}^p+\ip{\vs}^p\big]\dd|\cU_t|\dd t\;,
  \end{align*}
and hence $t\mapsto \ip{v}^p f_t$ is absolutely continuous in $L^1$.
\end{remark}
Next, we note that being a solution to the collision rate equation is
invariant under Maxwellian regularization.

Given $\mu\in\cP(\R^d)$, 
  we define its convolution with the
 Maxwellian $M$ as usual as the measure $\mu*M\in\cP(\R^d)$ given by
 \begin{align*}
   (\mu*M) (\dd v) = \int_{\R^d} M(v-w)\mu(\dd w)\dd v\;.
 \end{align*}
Given $\cU\in\cM(\R^{2d}\times S^{d-1})$ we define its
convolution $\cU*M$ with the Maxwellian $M$ in $\R^{2d}$ as the
measure given by
\begin{align*}
  (\cU*M)(\dd X,\dd\omega) = \int_{\R^{2d}}M(X-Y) \cU(\dd Y,\dd\omega)\dd X\;.
\end{align*}

\begin{lemma}\label{lem:ce-OU-invariance}
  Let $(\mu,\cU)\in\CE_T$ and set $\mu_t^\delta:=M_\delta*\mu_t$,
  $\cU^\delta_t:=M_\delta*\cU_t$ for $\delta\geq0$ and $t\in[0,T]$. Then we have
  $(\mu^\delta,\cU^\delta)\in\CE_T$.  
\end{lemma}

\begin{proof}
  Fix a test function $\phi$ and set
  $\Phi(X):=\phi(v)+\phi(\vs)$. Then, using
  \eqref{eq:commutation-scaling-convolution}, we find (dropping $\delta$ in the notation)
  \begin{align*}
    \ddt \int \phi \dd(\mu_t*M) &= \ddt \int (\phi*M)\dd\mu_t = \int \bar\nabla (\phi*M)\dd\cU_t\\
                                &= \int (\Phi*M)(T_\omega X)-(\Phi*M)(X) \dd\cU_t(X,\omega) \\
                                &= \int ((T_\omega\Phi)*M)(X)-(\Phi*M)(X) \dd\cU_t(X,\omega)\\
                                &= \int \Phi(T_\omega X)-\Phi(X)\dd(\cU_t*M)(X,\omega) = \int\bar\nabla\phi\dd(\cU_t*M)\;,
  \end{align*}
  which shows that $(\mu*M,\cU*M)\in \CE_T$.
\end{proof}

\subsection{The action functional}\label{sec:action}

Let us first recall the definition of the logarithmic mean $\Lambda:\R_+\times\R_+\to\R_+$ given by
\begin{align}\label{eq:log-mean}
  \Lambda(s,t)~=~\int_0^1s^\alpha t^{1-\alpha}\dd \alpha~=~\frac{s-t}{\log s-\log t}\ ,
\end{align}
the latter expression being valid for positive $s\neq t$. Note that
$\Lambda$ is concave and positively homogeneous, i.e.
$\Lambda(\a s,\a t)=\a\Lambda(s,t)$ for all $\a\geq0$. Moreover it is
easy to check that
\begin{equation}\label{eq:log-arith-ineq}
\Lambda(s,t)~\leq~\frac{s+t}{2}\quad\forall s,t\geq0\ .
\end{equation}

Given a function $f:\R^d\to\R_+$ we will often write
\begin{align*}
  \Lambda(f)(v,\vs,\omega)~=~\Lambda(f\fs,\fp\fsp)\;.
\end{align*}
We can now define a function $\alpha : \R_+\times\R_+\times\R \to[0,\infty]$ by setting
\begin{align}\label{eq:def-alpha}
 \alpha(s,t,u) :=  \left\{ \begin{array}{ll}
  \frac{u^2}{4\Lambda(s,t)}\;,
  & \Lambda(s,t) \neq 0\;,\\
0\;,
  &  \Lambda(s,t) = 0\text{ and } u = 0 \;,\\ 
+ \infty\;,
  & \Lambda(s,t) = 0 \text{ and } u \neq 0\;.\end{array} \right.
\end{align}
The function $\alpha$ is lower semicontinuous, convex and
  positively homo\-gen\-eous, i.e. for all $u\in\R$, $s,t\geq0$, and $r>0$ we have $\alpha(r s,r t,r u)~=~r \alpha(s,t,u)$. Indeed, this is easily checked using homogeneity and concavity of $\Lambda$
  and the convexity of the function $(u,y)\mapsto \frac{u^2}{y}$ on
  $\R \times (0,\infty)$.

  We will now define an action functional on pairs of measures
  $(\mu,\cU)$ where $\mu\in\cP(\R^d)$ and $\cU\in\cM(\Omega)$ generalizing
  \eqref{eq:action-pre}. For later reference, we work first in a more
  general setting.

  We consider the following integral functional associated with the
  function $\alpha$ on the space $\cM(X;\R^3)$ of vector-valued Borel
  measures with finite variation on a locally compact Polish space
  $X$ as defined in \eqref{eq:int-funct}, i.e.~we set:
  \begin{align}\label{eq:F-gen}
    \cF_\alpha(\lambda) := \int \alpha\left(\frac{\dd\lambda^1}{\dd|\lambda|},\frac{\dd\lambda^2}{\dd|\lambda|},\frac{\dd\lambda^3}{\dd|\lambda|}\right)\dd|\lambda|\;,
  \end{align}
where $|\lambda|$ denotes the variation of $\lambda$.
 
\begin{definition}[Action]\label{def:action}
 For $\mu\in\cP(\R^d)$ and $\cU\in\cM(\Omega)$ the \emph{action} is defined by
 \begin{align}\label{eq:action-def}
   \cA(\mu,\cU) := \cF_\alpha(\mu^1,\mu^2,\cU)\;,
 \end{align}
where $\mu^1,\mu^2$ are non-negative measures in $\cM_+(\Omega)$ given by
 \begin{align}\label{eq:defmu12}
  \mu^1(\dd v,\dd \vs,\dd \omega) := B(v-\vs,\omega)\mu(\dd v)\mu(\dd\vs)\dd\omega\;,\quad \mu^2 := T_\#\mu^1\;,
\end{align}
where $T$ is the change of variables
$(v,\vs,\omega)\mapsto (T_\omega(v,\vs),\omega)$ between pre- and
post-collisional variables defined in \eqref{eq:pre-post2}.
\end{definition}

If the measure $\mu$ is absolutely continuous w.r.t.~the Lebesgue
measure $\cL$ on $\R^d$, the next lemma shows that we recover
\eqref{eq:action-pre}. For this we denote by $\cB\in\cM(\Omega)$ the
measure given by
\begin{align*}
  \cB(\dd v,\dd\vs,\dd \omega)=B(v-\vs,\omega)\dd
  v\dd\vs\dd\omega\;.
\end{align*}

\begin{lemma}\label{lem:densities}
  Let $\mu=f\cL\in\cP(\R^d)$ and $\cU\in\cM(\Omega)$ be such that
  $\cA(\mu,\cU)<\infty$. Then there exists a Borel function $U:\Omega\to\R$ such
  that $\cU=U\Lambda(f) \cB$ and we have
  \begin{align}\label{eq:densities}
    \cA(\mu,\cU)~=~\frac14\int\abs{U(v,\vs,\omega)}^2\Lambda(f)B(v-\vs,\omega)\dd v\dd\vs\dd\omega\ .
  \end{align}
\end{lemma}

\begin{proof}
  Note that $\mu^i=\rho^i\cB,\ i=1,2$ with
  $\rho^1(v,\vs,\omega)=f(v)f(\vs)$ and
  $\rho^2(v,\vs,\omega)=f(\vp)f(\vsp)$. Choose $\sigma\in\cM(\Omega)$ such that $\cB=h\sigma$ and
  $\cU=\tilde U\sigma$ are both absolutely continuous
  w.r.t.~$\sigma$ and denote by
  $\tilde\rho^i$ the density of $\mu^i$ w.r.t $\sigma$. Now by
  homogeneity of $\alpha$
  \begin{align}\label{eq:densities1} 
   \cA(\mu,\cU)~=~\int\a\big(\tilde\rho^1,\tilde\rho^2,\tilde U\big)\dd\sigma~<~\infty\ .
  \end{align}
  Let $A\subset \Omega$ be such that $\int_A\Lambda(\rho^1,\rho^2)\dd
  \cB=0$. Homogeneity of $\Lambda$ yields
  \begin{align*}
   0~=~\int_A\Lambda(\rho^1,\rho^2)\dd \cB~=~\int_A\Lambda(\tilde\rho^1,\tilde\rho^2)\dd\sigma\ ,
  \end{align*}
  i.e.~$\Lambda(\tilde\rho^1,\tilde\rho^2)=0$ $\sigma$-a.e.~on
  $A$. Now the finiteness of the integral in \eqref{eq:densities1}
  implies that $\tilde U=0$ $\sigma$-a.e.~on $A$. Thus $|\cU|(A)=0$
  and hence $\cU$ is absolutely continuous w.r.t.~the measure
  $\Lambda(f)\cB$. Formula \eqref{eq:densities} now follows
  immediately from the homogeneity of $\a$.
\end{proof}

In view of the previous lemma, given a pair of functions
$f:\R^d\to\R_+$ and $U:\Omega\to\R$ we will define their action via
$\cA(f,U) := \cA(\mu,\cU)$ with $\mu=f\cL$ and
$\cU=U\Lambda(f)\tilde \cB$.

Next, we establish lower semicontinuity of the action
w.r.t.~convergence of $\mu$ and $\cU$. 

\begin{lemma}[Lower semicontinuity of the action]\label{lem:lscaction}
  Assume that $\mu_n\rightharpoonup\mu$ weakly in $\cP(\R^d)$ and
  $\cU_n\rightharpoonup^*\cU$ weakly* in $\cM(\Omega)$. Then
  \begin{align*}
    \cA(\mu,\cU)~\leq~\liminf\limits_{n}\cA(\mu_n,\cU_n)\ .
  \end{align*}
\end{lemma}

\begin{proof}
  Note that by the Assumption \ref{ass:kernel-bounded} on the collision
  kernel $B$, the weak convergence of $\mu_n$ to $\mu$ implies the
  weak* convergence of $\mu^i_n$ to $\mu^i$ in $\cM(\Omega)$ for
  $i=1,2$. Now the claim follows immediately from Lemma \ref{lem:Fprops}.
\end{proof}

The next estimate will be useful at several points in the paper. For
later reference, we formulate it in the general context of
\eqref{eq:F-gen}.

\begin{lemma}[Integrability estimate]\label{lem:integrability}
  For any Borel function $\Psi:X\to\R_+$ and
  $\lambda\in\cM(X;\R^3)$ with $\cF_\alpha(\lambda)<\infty$ and $\lambda^1,\lambda^2$ non-negative measures we have
  \begin{align}\label{eq:integrability-gen}
    \int \Psi\dd\abs{\lambda^3} \leq \sqrt{2\cF_\alpha(\lambda)} \left(\int \Psi^2 \dd (\lambda^1+\lambda^2)\right)^{\frac12}\;.
  \end{align}
\end{lemma}

\begin{proof}
  Let us write $\lambda^i=\rho^i|\lambda|$. Since
  $\cF_\alpha(\lambda)$ is finite, the set
  $A=\{\alpha(\rho^1,\rho^2,\rho^3)=\infty\}$ has zero measure with
  respect to $|\lambda|$. We can now estimate:
\begin{equation*}
\begin{split}
  &\int \Psi\dd \abs{\lambda^3}~\leq~\int \Psi\abs{\rho^3}\dd|\lambda| =~2\int_{A^c}\Psi \sqrt{\Lambda(\rho^1,\rho^2)}\sqrt{\a(\rho^1,\rho^2,\rho^3)}\dd|\lambda|\\
  &\leq~2\left(\int\limits\a(\rho^1,\rho^2,\rho^3)\dd|\lambda|\right)^{\frac{1}{2}}\left(\int_{A^c} \Psi^2\Lambda(\rho^1,\rho^2)\dd|\lambda|\right)^{\frac{1}{2}}\\
  &\leq~\sqrt{2\cF_\alpha(\lambda)}\left(\int \Psi^2\dd (\lambda^1+\lambda^2)\right)^{\frac12}\;,
 \end{split}  
\end{equation*}
where last inequality follows from the estimate
\eqref{eq:log-arith-ineq}.
\end{proof}

  \begin{corollary}\label{cor:more-integrability}
    Let $(\mu,\cU)\in\CE_T$ be such that
    $A:=\int_0^T\cA(\mu_t,\cU_t)\dd t$ and
    $E:=\int_0^T\cE_{2p+\gamma_+}(\mu_t)\dd t$ are finite for some $p>0$
    where $\gamma_+=\max(\gamma,0)$. Then the integrability condition
    \eqref{eq:more-integrability-U} is satisfied, precisely
  \begin{align*}
   \int_0^T \int \big[\ip{v}^p +\ip{\vs}^p\big]\dd\abs{\cU_t}\dd t
   \leq
   \sqrt{AC_BC_{p,\gamma}E}\;.
  \end{align*}
\end{corollary}

\begin{proof}
  Let $\mu^i,\cU\in\cM(\Omega\times[0,T])$ be given by $\dd\mu^i=\dd\mu^i_t\dd t$ and $\dd\cU=\dd\cU_t\dd t$ and note that
  \begin{align*}
    \int_0^T\cA(\mu_t,\cU_t)\dd t = \int_0^T\cF_\alpha(\mu^1_t,\mu^2_t,\cU_t)\dd t  =\cF_\alpha(\mu^1,\mu^2,\cU)\;.
  \end{align*}
  Then, one concludes by Lemma \ref{lem:integrability}, choosing $\Psi(v,\vs,\omega,t)= \ip{v}^p+\ip{\vs}^p$.
\end{proof}

Note that for a given curve $(\mu_t)_{t\in[0,T]}$ there will be
several compatible collisions rates $(\cU_t)_t$ such that
$(\mu,\cU)\in\CE_T$. For instance, when $\cV_t$ is symmetric under the
transformation $(v,\vs,\omega)\mapsto(\vp,\vsp,\omega)$ we have
$\int\bar\nabla\phi\dd\cV_t=0$ for any test function $\phi$. Hence,
$(\mu,\cU+\cV)\in\CE_T$ whenever $(\mu,\cU)\in\CE_T$. Thus, we define
the action of a curve as the minimal action of all compatible
collision rates.

\begin{definition}[Action of a curve]\label{def:action-curve}
  Given a curve $(\mu_t)_{t\in[0,T]}$ in $\cP(\R^d)$ its \emph{action}
  is defined by
  \begin{align}\label{eq:action-curve-def}
    \cA_T(\mu) := \inf\left\{\int_0^T\cA(\mu_t,\cU_t)\dd t~:~(\mu,\cU)\in\CE_T\right\}\;.
  \end{align}
  If there is no $\cU$ with $(\mu,\cU)\in\CE_T$, we set $\cA_T(\mu)=+\infty$.
\end{definition}

The next result shows that under additional control on the energy of
the curve the infimum above is attained by an optimal collision rate.

\begin{proposition}[Optimal collision rate]\label{prop:opt-collision-rate}
  Let $(\mu_t)_{t\in[0,T]}$ be a curve in $\cP(\R^d)$ such that
  \begin{align}\label{eq:act-energy-bdd}
   \cA_T(\mu)<\infty\;,\qquad E:=\int_0^T\cE_{2}(\mu_t)\dd t<\infty\;. 
  \end{align}
  Then, there exists a family $(\cU_t)_{t}$ with
  $(\mu,\cU)\in\CE_T$ attaining the infimum in
  \eqref{eq:action-curve-def}.
\end{proposition}

\begin{proof}
  Let $(\cU^n_t)_t$ be a minimizing sequence of collision rates for
  \eqref{eq:action-curve-def} and define the measures
  $\cU^n\in\cM(\Omega\times[0,T])$ given by $\dd\cU^n=\dd\cU^n_t\dd t$. By
  Lemma \ref{lem:integrability}, for every measurable function $\Psi$ on
  $\R^{2d}\times S^{d-1}\times [0,T]$ we have
  \begin{align}\label{eq:unif-integrability3}
    &\sup\limits_n\int \Psi\dd \abs{\cU^n}\\\nonumber
&\leq\sqrt{2A}\left(\int\big(\Psi^2+\Psi^2\circ T)B(v-\vs,\omega)\dd\omega\dd\mu_t(v)\dd\mu_t(\vs)\dd t\right)^{\frac12}\;,
  \end{align}
  with $A=\sup_n\int_0^T\cA(\mu_t,\cU^n_t)\dd t<\infty$. Choosing
  $\Psi=\one_{\Omega\times I}$ and using Assumption
  \ref{ass:kernel-bounded}, we obtain
  $\abs{\cU^n}(\Omega\times I)\leq 2\sqrt{C_B A E\cdot \cL(I)}$.  Hence,
  $\cU^n$ has uniformly bounded variation and we have that up to
  extracting a subsequence $\cU^n\rightharpoonup^*\cU$ in
  $\cM(\Omega\times[0,T])$. Moreover, we see that $\cU$ can be
  disintegrated w.r.t.~Lebesgue measure on $[0,T]$ and we can write
  $\cU=\int_0^T\cU_t\dd t$ for a Borel family $(\cU_t)$ still
  satisfying (iii) in Definition~\ref{def:ce}.

  To see that $(\mu,\cU)\in\CE_T$, it suffices to show that for any
  test functions $a\in C\big([0,T]\big)$ and $\phi\in C_b(\R^d)$ we have
   \begin{align}\label{eq:converge-bnu2}
    \int a(t)\dgrad\phi \dd\cU^n_t
    \dd t~\overset{n\to\infty}{\longrightarrow}~\int a(t)\dgrad\phi
    \dd\cU_t \dd t\;.
  \end{align}
  This follows from a straightforward argument, approximating
  $\dgrad \phi$ with compactly supported continuous functions $\Omega$ once
  we establish the following tightness estimate for $\cU^n$:
  Denoting by $B_R$ the ball of radius $R$ in $\R^{2d}$ and $M_R:=B_R^c\times S^{d-1}\times[0,T]$ we have
  \begin{align*}
  \abs{\cU^n}(M_R)
  \leq
  2\sqrt{AC_B}\left(\int_{0}^{T}\int_{B_{R/2}^c}\ip{v}^\gamma+\ip{\vs}^\gamma\dd\mu_t(v)\dd\mu_t(\vs)\dd t\right)^{\frac12}\leq~2\frac{\sqrt{AC_BE}}{\sqrt{R}}\;,
  \end{align*}
  which goes to zero uniformly in $n$ as $R\to\infty$.  This estimate
  follows again from \eqref{eq:unif-integrability3}, noting that if
  $(v,\vs)$ or $(\vp,\vsp)$ lies outside $B_R$, then $(v,\vs)$ lies
  outside of $B_{R/2}$, and further using the estimate
  $\int_{\{|v|\geq R\}}\ip{v}^\gamma\dd\mu_t(v)\leq \int
  \frac{\ip{v}^2}{R}\dd\mu_t(v)$, since $\gamma\leq 1$, and the upper
  bound on the energy in \eqref{eq:act-energy-bdd}.  Finally, we
  conclude that $\int_0^T\cA(\mu_t,\cU_t)\dd t=\cA(\mu)$ noting that
  $\int_0^T\cA(\mu_t,\cU_t)\dd t=\cF_\alpha(\mu^1,\mu^2,\cU)$ and
  using lower semicontinuity of $\cF_\alpha$.
\end{proof}

\section{Variational characterization of the homogeneous Boltzmann
  equation}
\label{sec:gradflow}

In this section we establish the variational characterization of the
homogeneous Boltzmann equation stated in Theorem \ref{thm:EDI-intro}. The crucial ingredient is a chain rule allowing to take derivatives of the entropy along suitable curves of finite action.

Recall that $\cE_p(\mu)$ denotes the $p$-moment of $\mu$.

\begin{proposition}[Chain rule]\label{prop:chainrule}
  Let $(\mu,\cU)\in\CE_T$ with $(\mu_t)_t\subset \mathcal P_p(\mathbb R^d)$ such that $\cH(\mu_t)$ is finite for some $t\in[0,T]$,
  \begin{align*}
   E:= \int_0^T\cE_{p}(\mu_t)\dd t<\infty\;,
  \end{align*}
  where $p=2+\max(\gamma,0)$ and
  \begin{align}\label{eq:finite-D-A}
    \int_0^T\sqrt{\cA(\mu_t,\cU_t)}\dd t < \infty\;,\quad
\int_0^T\sqrt{D(\mu_t)}\sqrt{\cA(\mu_t,\cU_t)}\dd t < \infty\;.
  \end{align}
  Then $\cH(\mu_t)<\infty$ for all $t\in[0,T]$ and we have that
  \begin{align}\label{eq:chainrule}
    \cH(\mu_t)-\cH(\mu_s) = \int_s^t\frac14\int_{\{\Lambda(f_{r}>0\}}\dgrad\log f_r \dd\cU_r\dd r\quad\forall 0\leq s\leq t\leq T\;,
  \end{align}
  where $f_r$ is the density of $\mu_r$. In particular, the map
  $t\mapsto \cH(\mu_t)$ is absolutely continuous and we have
  \begin{align}\label{eq:chainrule-diff}
    \frac{\dd}{\dd t}\cH(\mu_t) = \frac14\int\dgrad\log f_t \dd\cU_t\quad\text{ for a.e.~}t\;.
  \end{align}
\end{proposition}

Note that the assumption \eqref{eq:finite-D-A} and Lemma \ref{lem:densities} imply that for a.e.~$t$ $\mu_{t}$ is absolutely continuous with a density $f_{t}$ and $\cU_{t}$ is absolutely continuous with a density $U_{t}\Lambda(f_{t})B$, in particular the set of $(v,\vs,\omega)$ where $\Lambda(f_{t})=0$ is negligible for $\cU_{t}$. Hence the right hand side in \eqref{eq:chainrule} is well-defined since $f,\fs,\fp,\fsp>0$ on $\{\Lambda(f)>0\}$. More precisely, this and similar integrals in the sequel will be understood implicitly to be taken over the set $\{\Lambda(f_{t})>0)\}$.

As a preparatory result we establish the following continuity property
of the action and dissipation under Maxwellian convolution.

\begin{lemma}\label{lem:AD-conv}
  Let $\mu\in \cP(\R^d)$ and $\cU\in \cM(\Omega)$ such that $\cA(\mu,\cU)<\infty$ and $D(\mu)<\infty$. Let $\mu^\delta=M_\delta*\mu$ and $\cU^\delta=M_\delta*\cU$ denote convolution with the Maxwellian. Then we have
  \begin{equation}
    \label{eq:AD-conv}
    \lim_{\delta\to0}\cA(\mu^\delta,\cU^\delta)=\cA(\mu,\cU)\;,\qquad \lim_{\delta\to0}D(\mu^\delta)=D(\mu)\;.
  \end{equation}
  Moreover, there is a constant $C$ depending only on $\gamma$ and $c_B$
  from Assumption \ref{ass:kernel-bounded} such that
  \begin{equation}
    \label{eq:AD-conv-bound}
    \cA(\mu^\delta,\cU^\delta)\leq C\cA(\mu,\cU)\;,\quad D(\mu^\delta)\leq C D(\mu)\quad \forall \delta>0\;.
  \end{equation}
\end{lemma}

In the proofs of the latter two results, we took inspiration from
\cite{CDDW20} to treat the case of unbounded kernels $B$. Namely in
the usage of the Petree inequality and the following version of the
dominated convergence theorem (see e.g.~\cite[Chap.~4,
Thm.~17]{Roy88}), termed \emph{extended dominated convergence
  theorem}: Let $(R_\delta)_{\delta>0}$ and $(I_\delta)_{\delta>0}$ be
families of measurable functions on a measure space $X$ with
$I_\delta\geq0$ and let $R,I$ be measurable. Assume that
$R^\delta,I^\delta$ converge point-wise to $R,I$ respectively,
$|R^\delta|\leq I^\delta$ a.e., and
$\lim_{\delta\to0}\int_XI^\delta=\int_XI$. Then we also have
$\lim_{\delta\to0}\int_X R^\delta=\int_XR$.

\begin{proof}[Proof of Lemma \ref{lem:AD-conv}]
  We first prove \eqref{eq:AD-conv}, \eqref{eq:AD-conv-bound} for the
  dissipation $D$. Let us $\mu=f\dd v$ and $\cU=U\dd X \dd\omega$ and
  put $F(X)=ff_*$. Similarly let $f^\delta=M_\delta*f$,
  $F^\delta=M_\delta*F$, $U^\delta=M_\delta*U$ be the densities of
  $\mu^\delta$, $\mu^\delta\otimes\mu^\delta_*$ and $\cU^\delta$. Now,
  \begin{align*}
    D(\mu^\delta)=\int B\Lambda(f^\delta)|\bar\nabla f^\delta|^2\dd X\dd\omega
    = \int B G(F^\delta,T_\omega F^\delta)\dd X\dd \omega=:\int L^\delta_1\dd X\dd\omega\;,
  \end{align*}
  where $G(x,y)=(x-y)\big(\log x-\log y)$ is convex. Note that
  $L^\delta_1(X,\omega)$ converges point-wise to
  $L(X,\omega):=B\cdot G(F,T_\omega F)$ as $\delta\to0$ and
  $\int L\dd X\dd\omega=D(\mu)$. From the commutation property of
  Lemma \ref{lem:OU-commutation} and Jensen's inequality we infer the
  majorant
      \begin{align*}
        L^\delta_1\leq B\cdot\big( M_\delta*\big[G(F,T_\omega F)\big]\big)=:L^\delta_2\;.
      \end{align*}
      Obviously also $L^\delta_2\to L$ point-wise as $\delta\to0$. To
      prove the continuity \eqref{eq:AD-conv} it suffices by the
      extended dominated convergence theorem to show that
      $\int L^\delta_2\dd X\dd\omega\to \int L\dd X\dd\omega$. But by self-adjointness of the convolution  we have that
      \begin{align*}
        \int L^\delta_2\dd X\dd\omega=\int B\cdot\big( M_\delta *\big[G(F,T_\omega F)\big]\big)\dd X\dd\omega
        = \int \big(M_\delta * B\big)\cdot G(F,T_\omega F)\dd X\dd\omega:=\int L^\delta_3\dd X\dd\omega\;,
      \end{align*}
      and again $L^\delta_3\to L$. Now, by Assumption \ref{ass:kernel-bounded} and Lemma \ref{lem:kernel-conv} we have
      \begin{align*}
        \big(M_\delta * B\big)(X,\omega)&\leq c_B \int \ip{v-v_*-(w-w_*)}^\gamma M_\delta(w)M_\delta(w_*)\dd w\dd w_*\\ &\leq C_\gamma c_B \ip{v-v_*}^\gamma\leq C_\gamma c_B^2 B(X,\omega)\;.
      \end{align*}
      Hence, we have a majorant $L^\delta_3\leq C L$ and dominated
      convergence yields
      $\int L^\delta_2\dd X\dd\omega=\int L^\delta_3\dd X\dd\omega \to
      \int L\dd X\dd\omega$ as desired. Note that the previous
      argument also yields the bound \eqref{eq:AD-conv-bound}.

      To prove the corresponding claims for the action $\cA$, we proceed in the same way,
      writing
      \begin{align*}
        \cA(\mu^\delta,\cU^\delta) = \int B^{-1}\frac{|U^\delta|^2}{\Lambda(F^\delta,T_\omega F^\delta)}\dd X\dd\omega\;,
      \end{align*}
      and use convexity of the function $(u,r,s)\to|u|^2/\Lambda(r,s)$ as well as in the last step that similarly $M^\delta *B^{-1}\leq C B^{-1}$. 
    \end{proof}

\begin{proof}[Proof of Proposition \ref{prop:chainrule}]
  Note that by \eqref{eq:finite-D-A} and Lemma \ref{lem:densities} we
  have $\mu_r=f_r\dd v$, $\cU_r=U_r\dd X\dd\omega$ for a.e.~$r$ and
  suitable densities $f_r,U_r$. We will now proceed in several steps.

{\it Step 1: Regularization.}

We will perform a three-fold regularization procedure. First, we
regularize the curve by convolution with the Maxwellian. For
$\delta>0$ we set $\mu_t^\delta=M_\delta*\mu_t$, and
$\cU_t^\delta=M_\delta*\cU_t$. Then we perform a convolution in
time. For a standard mollifier $\eta$ on $\R$ supported in $[-1,1]$
and $\lambda>0$ we define
  \begin{align*}
    \mu^{\delta,\lambda}_t = \int \eta(t')\mu^{\delta}_{t-\lambda t'}\dd t'\;,\quad 
       \cU^{\delta,\lambda}_t = \int \eta(t')\cU^{\delta}_{t-\lambda t'}\dd t'\;.
  \end{align*}
  (For this the curves are assumed to be extended trivially by
  $\mu^\delta_0,\cU^\delta_0$ on $[-\lambda,0]$ and similarly on
  $[T,T+\lambda]$.) By Lemma \ref{lem:ce-OU-invariance} we
  have that $(\mu^\delta,\cU^\delta)\in\CE_T$ and by linearity of
  the collision rate equation also
  $(\mu^{\delta,\lambda},\cU^{\delta,\lambda})\in\CE_T$.

  Finally, let $g$ be a probability density in $\cPE$ such that
  \begin{align}\label{eq:bound-g}
    \abs{\log g(v)}\leq C\ip{v}\;,
  \end{align}
  for some constant $C$ (for instance choose $g(v)$ proportional to
  $e^{-\alpha|v|}$ for suitable $\alpha>0$). Then we set for $\eps>0$,
  $\mu^{\delta,\lambda,\eps}:=(1+\eps)^{-1}(\mu^{\delta,\lambda}+\eps
  g\cL)$,
  and $\cU^{\delta,\lambda,\eps}=(1+\eps)^{-1}\cU^{\delta,\lambda}$ and
  note that
  $(\mu^{\delta,\lambda,\eps},\cU^{\delta,\lambda,\eps})\in\CE_T$. Let
  $f^\delta,U^\delta$ denote the densities of $\mu^\delta,\cU^\delta$
  and similarly with $\lambda$ and $\eps$.  \smallskip

  {\it Step 2: Estimates for the regularized curve.}

  Note that that the time-integrated $p$-moment of $\mu^{\delta}_r$ is bounded as
  \begin{align}\label{eq:moment-mu}
   \int_0^T \cE_{p}(\mu_r^{\delta,\lambda,\eps})\dd r\leq E\;,
  \end{align}
  with $p=2+\max(\gamma,0)$ for all $\delta,\lambda,\eps>0$.
  
  Next, we look at the behavior of the action and dissipation under
  the regularization. From Lemma \ref{lem:AD-conv} we have
    \begin{align}
      \label{eq:act-diss-2}
\cA(\mu^\delta,\cU^\delta)&\leq C\cA(\mu,\cU)\;,\quad D(\mu^\delta)\leq C D(\mu)\;.
    \end{align}
     A similar convexity argument gives that
     \begin{align}\label{eq:action-finite}
       \int_0^T\cA(\mu^{\delta,\lambda}_r,\cU^{\delta,\lambda}_r)\dd r \leq C \int_0^T\cA(\mu_r,\cU_r)\dd r\;.
     \end{align}
     Taking into account Corollary \ref{cor:more-integrability} and
     \eqref{eq:moment-mu} we obtain
      \begin{align}\label{eq:int-U-unif}
        \int_0^T\int \big[\ip{v} + \ip{\vs}\big]\dd|\cU_r^{\delta,\lambda}|\dd r\leq C\;,
      \end{align}
      uniformly in $\delta,\lambda>0$.
  \smallskip
  
 {\it Step 3: Integrated chain rule for regularized curve.}

 Now, we claim that 
  \begin{align}\label{eq:time-deriv}
    \frac{\dd}{\dd r}\cH(\mu^{\delta,\lambda,\eps}_r)
    &= 
    \int_{\R^d} \log f^{\delta,\lambda,\eps}_r \partial_rf^{\delta,\lambda,\eps}_r 
    = 
   \frac14\int_\Omega \dgrad \log f^{\delta,\lambda,\eps}_rU^{\delta,\lambda,\eps}_r\;,
  \end{align}
  where the integral over $\Omega$ is w.r.t.~the measure $\dd
  X\dd\omega$. Indeed, to justify the first identity in \eqref{eq:time-deriv} we use convexity of $r\mapsto r\log r$ and \eqref{eq:bound-g} to estimate
  \begin{align*}
&\frac{1}{h}\big|f^{\delta,\lambda,\eps}_{r+h}\log f^{\delta,\lambda,\eps}_{r+h} - f^{\delta,\lambda,\eps}_r\log f^{\delta,\lambda,\eps}_{r} \big|
\leq C\ip{v}
   \frac{1}{h}\big|f^{\delta,\lambda,\eps}_{r+h} - f^{\delta,\lambda,\eps}_r \big| \\
\leq &C\ip{v} \norm{\eta'}_{\infty}\int_0^T\big(f^{\delta}_{t}+\eps g\big)\dd t\;.
   \end{align*}
   Since $(f_t)_t$ has uniformly bounded time-integrated second moment, by dominated convergence we
   can take the time derivative inside the integral. The second
   identity in \eqref{eq:time-deriv} follows by applying the collision
   rate equation using \eqref{eq:bound-g} and \eqref{eq:int-U-unif}, see Remark
   \ref{rem:generaltest}.

  Integrating \eqref{eq:time-deriv} between $s$ and $t$ we obtain
  \begin{align}\label{eq:chain-reg}
    \cH(f^{\delta,\lambda,\eps}_t) - \cH(f^{\delta,\lambda,\eps}_s)
    = \int_s^t \frac14\int_\Omega \dgrad \log f^{\delta,\lambda,\eps}_rU^{\delta,\lambda,\eps}_r\dd r\;.
  \end{align}
\smallskip

{\it Step 4: Passing to the limit.}

We will now pass to the limit in \eqref{eq:chain-reg} to obtain
\eqref{eq:chainrule} letting $\lambda\to0$, $\eps\to0$ and $\delta\to0$
in this order. Consider first the right hand side.  \smallskip

{\it a) RHS, $\lambda\to0$.}

Using the bound
$|\log f^{\delta,\lambda,\eps}_r|\leq c(\delta,\eps)\ip{v}$ ensured by
\eqref{eq:bound-g} which is uniform in $\lambda$ for fixed
$\delta,\eps$ and the integrability condition \eqref{eq:int-U-unif}
for $U^{\delta,\lambda}$, we can pass to the limit as $\lambda\to0$ and
obtain
  \begin{align}\label{eq:RHS1}
    (1+\eps)^{-1}\int_s^t \frac14\int_\Omega \dgrad \log (f^{\delta}_r+\eps g)U^{\delta}_r\dd r\;.
  \end{align}
\smallskip

{\it b) RHS, $\eps\to0$.}

  We use the estimate
  (dropping time parameter $r$ in the notation):
  \begin{align}\nonumber
    &|\dgrad \log (f^\delta+\eps g) U^\delta|\leq 
\sqrt{|\dgrad \log (f^{\delta}+\eps g)|^2\Lambda(f^{\delta}+\eps g)B}\cdot\sqrt{\frac{|U^\delta|^2}{B\Lambda(f^{\delta}+\eps g)}}\\\nonumber
  &\leq\sqrt{ \Big|\dgrad \log (f^{\delta}+\eps g)\Big|\cdot\Big|(f^{\delta}+\eps g)\big((f^\delta)_* +\eps g_*\big)-\big((f^{\delta})'+\eps g'\big)\big((f^\delta)'_* +\eps g'_*\big)\Big|B}\\\nonumber
&\quad  \cdot\sqrt{\frac{|U^\delta|^2}{B\Lambda(f^\delta)}}\\\nonumber
  &\leq \sqrt{C\big(\ip{v}^2+\ip{\vs}^2\big)\ip{v-\vs}^\gamma \cdot\Big|(f^{\delta}+\eps g)\big((f^\delta)_* +\eps g_*\big)-\big((f^{\delta})'+\eps g'\big)\big((f^\delta)'_* +\eps g'_*\big)\Big|}
    \\\label{eq:2terms}
&\quad \cdot\sqrt{\frac{|U^\delta|^2}{B\Lambda(f^\delta)}}\;.
\end{align}
Here, in the second inequality we have used the definition of
$\Lambda$ and the monotonicity of the logarithmic mean. In the third
inequality we used the bound \eqref{eq:Gaussian-bound} and Assumption
\ref{ass:kernel-bounded}. By the moment assumptions on $f$ and Lemma
\ref{lem:AD-conv}, one readily checks that the right-hand side in   
\eqref{eq:2terms} is integrable on $[0,T]\times \Omega$ and its integral converges to that of the same expression with $\eps=0$. Thus, the extended dominated convergence theorem allows us to pass to the limit
as $\eps \to 0$ in \eqref{eq:RHS1} and obtain
  \begin{align}\label{eq:RHS2}
    \int_s^t \frac14\int,\dgrad \log f^{\delta}_rU^{\delta}_r\dd r\;.
  \end{align}
\smallskip

{\it c) RHS, $\delta\to0$.}

  Note that $\dgrad \log f^{\delta}_rU^{\delta}_r$ converges point-wise
  to $\dgrad \log f_rU_r$ as $\delta\to0$ at every $r$ where the
  densities of $\mu_r,\cU_r$ exist. To pass to the limit in the
  integral over $\Omega$ we use the majorant (dropping the time parameter $r$
  in the notation)
  \begin{align*}
    &|\dgrad \log f^{\delta}U^{\delta}|
    \leq    \frac12|\dgrad\log f^\delta|^2\Lambda(f^\delta)B+\frac12 \frac{|U^\delta|^2}{B\Lambda(f^\delta)} =:\frac12(I_1^\delta+I_2^\delta)\;.
  \end{align*}
  Obviously $I_1^\delta\to I^0_1$ and $I_2^\delta\to I^0_2$ point-wise,
  where $I_1^0$ and $I_2^0$ are the corresponding expressions with
  $f^\delta$ and $U^\delta$ replaced by $f, U$. By Lemma \ref{lem:AD-conv}, we also have
  \[\int I^\delta_1 =D(f^\delta)\to D(f) = \int I^0_1\;,\quad \int I^\delta_2 =\cA(f^\delta,U^\delta)\to \cA(f,U) = \int I^0_2\;, \]
  as $\delta\to0$ for a.e.~$r\in [0,T]$. Thus by the extended
  dominated convergence theorem we can pass to the limit in the space
  integral in \eqref{eq:RHS2}.

  Finally, to pass to the limit in the time integral, we use
  the already established almost everywhere in time convergence of the
  space integral and exhibit a majorant similar as above and using Lemma \ref{lem:AD-conv}:
  \begin{align*}
   &\int \dgrad \log f^{\delta}_rU^{\delta}_r\dd r
    \leq \left(\int I^\delta_1\right)^{\frac12}
        \left(\int I^\delta_2\right)^{\frac12}
     \leq C\sqrt{D(\mu_r)}\sqrt{\cA(\mu_r,\cU_r)}\;.
  \end{align*}
  Recall that the last expression is integrable by assumption.
\smallskip

{\it d) LHS.}

  Let us turn to show convergence of the left hand side of
  \eqref{eq:chain-reg}. Appealing to the bound
  \eqref{eq:bound-g} for $g$ we obtain the estimate
  \begin{align*}
    \abs{\cH(f^{\delta,\lambda,\eps}_t)-\cH(f^{\delta,\eps}_t)}
    \leq
    C\int\ip{v}\big|f^{\delta,\eps}_{t-\lambda t'}-f^{\delta,\eps}_t\big|\eta(t')\dd t'\;,
  \end{align*}
  and we can pass to the limit as $\lambda\to0$ by the continuity of
  $t\mapsto \ip{v}f^{\delta,\eps}_t$ in $L^1$, see Remark
  \ref{rem:generaltest}. The bound \eqref{eq:Gaussian-bound} allows to
  pass to the limit as $\eps\to0$ and we are left with
  $\cH(f^\delta_t)-\cH(f^\delta_s)$. Assume first that $\cH(\mu_s)$ is
  finite. Recall that entropy is decreasing along the
  Ornstein--Uhlenbeck semigroup and lower semicontinuous. As
  $\delta\to0$ we thus have that $\cH(f^\delta_t)$ increases to
  $\cH(\mu_t)$. Thus, $\cH(f^\delta_t)-\cH(f^\delta_s)$ converges to
  $\cH(f_t)-\cH(f_s)$ and $\cH(\mu_t)$ is finite due to the
  boundedness of the right hand side of \eqref{eq:chainrule} in the
  limit. Since by assumption there exists $s$ with
  $\cH(\mu_s)<\infty$, this shows that $\cH(\mu_t)<\infty$ for all
  $t\in[0,T]$ and \eqref{eq:chainrule} is established.
  
  Finally, using the estimate
  \begin{align}\label{eq:bound-DA}
    \frac14\int \dgrad\log f_r \dd\cU_r
   &\leq \sqrt{D(\mu_r)}\sqrt{\cA(\mu_r,\cU_r)}\;,
  \end{align}
  that is obtained just as the one before for $f^\delta_r$ we see that
  $t\mapsto\cH(\mu_t)$ is absolutely continuous and
  \eqref{eq:chainrule-diff} follows.
 \end{proof}

We can now prove the variational characterization of the homogeneous
Boltzmann equation as the gradient flow of the entropy. For
convenience we rephrase the statement here.

By a solution to the homogeneous Boltzmann equation we mean a family
of probability densities $(f_t)_{t\geq 0}$ such that
$f\in C\big([0,\infty);L^1(\R^d)\big)\cap
L^\infty\big([0,\infty);L^1_2(\R^d)\big)$ such that for all
$\phi\in C^\infty_c(\R^d)$ in distribution sense:
\begin{align}\label{eq:weakBE}
   \ddt \int_{\R^d} \phi f_t = -\frac14\int\dgrad
  \phi (\fp\fsp-f\fs)B(v-\vs,\omega)\dd v\dd\vs\dd\omega\;.
\end{align}

\begin{theorem}\label{thm:EDI}
  Let $(f_t)_{t\in[0,T]}$ be a curve of probability
  densities in $\mathcal P_p(\mathbb R^d)$ such that
 \begin{align}\label{eq:f-not}
  \cH(f_0)<\infty\;, \qquad  \int_0^T\cE_{p}(f_t)\dd t<\infty\;. 
 \end{align}
 with $p=2+\max(\gamma,0)$. Then we have that:
 \begin{align*}
   J_T(f):=\cH(f_T)-\cH(f_0)+\frac12\int_0^TD(f_t)\dd t +\frac12\cA_T(f) \geq 0 \;.
 \end{align*}
 Moreover, we have $J_T(f)=0$ if and only if $(f_t)_t$ is solution to
 the homogeneous Boltzmann equation satisfying the integrability
 assumptions \eqref{eq:f-not} and
 \begin{align}\label{eq:dissipation-integrable}
   \int_0^T D(f_t)\dd t < \infty\;.
 \end{align}
\end{theorem}

Assuming finite entropy and energy (and finite fourth moment for
$\gamma>0$) of the initial datum $f_0$, Theorem \ref{thm:boltzmann}
gives existence and uniqueness of a solution $(f_t)_t$ to the
homogeneous Boltzmann equation. It satisfies \eqref{eq:f-not} and
\eqref{eq:energy-identity}, in particular,
\eqref{eq:dissipation-integrable} holds. Thus, there is actually only
one curve such that $J_T(f)=0$, namely the unique solution to the
Boltzmann equation.

\begin{proof}[Proof of Theorem \eqref{thm:EDI}]
  Let $(f_t)_{t\in[0,T]}$ be a curve satisfying \eqref{eq:f-not}. To
  show $J_T(f)\geq0$ we can assume that $\cA_T(f)<\infty$ and
  $\int_0^TD(f_t)\dd t<\infty$, since otherwise $J_T(f)=+\infty$. Let
  $(\cU_t)_t$ be optimal collision rates given by Proposition
  \ref{prop:opt-collision-rate}. But then $J_T(f)\geq0$ follows
  immediately from Proposition \ref{prop:chainrule} and the estimate
  \eqref{eq:bound-DA}.

  We now show that any solution $(f_t)$ to the Boltzmann equation satisfying
  \eqref{eq:f-not} and \eqref{eq:dissipation-integrable} satisfies
  $J_T(f)=0$. Setting $\mu_t=f_t\cL$ and
 \begin{align*}
   \cU_t=-\dgrad\log f_t\Lambda(f_t)\cB = -\big[(\fp)_t(\fsp)_t-f_t(\fs)_t\big]\cB\;,
 \end{align*}
 we see by \eqref{eq:weakBE} that $(\mu,\cU)$ belongs to $\CE$. Moreover, we have that
 $\cA(\mu_t,\cU_t)=D(f_t)$ and thus by
 \eqref{eq:dissipation-integrable} we can apply the chain rule
 \eqref{eq:chainrule} to obtain
 \begin{align*}
   \cH(f_T) - \cH(f_0)= -\int_0^T D(f_t)\dd t = -\frac12 \int_0^T D(f_t) \dd t-\frac12 \cA_T(\mu)\;,
 \end{align*}
 i.e.~$J_T(f)=0$.

 Conversely, let us show that any curve $(f_t)_t$ with $J_T(f)=0$ is a
 solution to the Boltzmann equation satisfying \eqref{eq:dissipation-integrable}. From
 \eqref{eq:f-not} we obtain that $\cH(\mu_t)<\infty$ for all $t$ and
 that $\cA_T(f)<\infty$ and \eqref{eq:dissipation-integrable}
 holds. By Proposition \ref{prop:opt-collision-rate} there exists a
 family $\cU_t$ with $(\mu,\cU)\in\CE_T$ such that
 $\int_0^T\cA(f_t,\cU_t)\dd t=\cA_T(f)$, in particular $t\mapsto f_t$
 is continuous in $L^1$, see Remark \ref{rem:refined}. By Lemma \ref{lem:densities} the measure
 $\cU_t$ has a density $U_t\Lambda(f_t)B$. From the chain rule
 \eqref{eq:chainrule} and the Cauchy--Schwartz and Young inequalities
 we infer that
 \begin{align*}
   \cH(f_T)-\cH(f_0) &= \int_0^T\frac14\int\dgrad\log f_r U_r \Lambda(f_r)B\dd r\\
                        &\geq -\int_0^T\left[\sqrt{\frac14\int |\dgrad\log f_r|^2\Lambda(f_r)B}\sqrt{\frac14\int |U_r|^2\Lambda(f_r)B}\right]\dd r\\
                        &\geq -\frac12 \int_0^T\left[\frac14\int|\dgrad\log f_r|^2\Lambda(f_r)B + \frac14\int|U_r|^2\Lambda(f_r)B\right] \dd r\\
                        &= -\frac12\int_0^T D(f_r)\dd r- \frac12\cA_T(f)\;.
 \end{align*}
 Since $J_T(f)=0$, we see that the two inequalities have to be
 identities. This implies that
 $\int_0^T\int|U_r+\dgrad\log f_r|^2\Lambda(f_r)B\dd t=0$, hence
 $U_r=-\dgrad\log f_r$ for a.e.~$r$ and a.e.~$(v,\vs,\omega)$ with
 $\Lambda(f_r)(v,\vs,\omega)>0$. Thus, the collision rate equation for
 $(\mu,\cU)$ turns into the distributional formulation of the Boltzmann
 equation.
\end{proof}

The results obtained in this section can be recast in the framework of
gradient flows in metric spaces. The action functional gives rise to a
distance $\cW_B$ on $\cP_{p,E}(\R^d)$ and the Boltzmann equation is
characterized as the gradient flow of $\cH$ (i.e.~a curve of maximal
slope) in the metric space $(\cP_{p,E}(\R^d),\cW_B)$. We refer to
the appendix for a discussion of this point of view, in particular to
Corollary \ref{cor:sug-dissipation}.

\subsection*{Generalized gradient structures}

We will now briefly discuss possible generalizations of the variational characterization of the Boltzmann equation above using generalized gradient structures. Such structures arise naturally from the large deviation behavior of an underlying microscopic stochastic system, c.f.~the discussion in the introduction. We refer the reader e.g.~to \cite{MPR14} and the references therein. Here we aim to indicate how the characterization in Theorem \ref{thm:EDI} can be generalized to non-quadratic gradient structures. A very detailed discussion of generalized gradient structures associated in the case of jump processes has recently been performed in \cite{PRST20}. We will mainly follow their terminology and constructions and adapt them to the present case of the Boltzmann equation. In comparison, we will impose more restrictive conditions on the gradient structures we consider in order to simplify the presentation while still encompassing the main examples we are interested in, see Example \ref{ex:quad-gen}.

Let us fix a \emph{dual dissipation density} $\Psi^*$ and a \emph{flux density map} $\theta$ as follows. 

\begin{assumption}\label{ass:gen-grad}
 We assume that
\begin{enumerate}
\item the function $\Psi^*:\R\to[0,\infty)$ is convex, differentiable, superlinear and even with $\Psi^*(0)=0$; 
\item the function $\theta:[0,\infty)\times[0,\infty)\to[0,\infty)$ is assumed to be continuous, concave, and not identically $0$ and satisfies
\begin{itemize}
\item symmetry: $\theta(r,s)=\theta(s,r)$ for all $s,r\in [0,\infty)$;
\item positive $1$-homogeneity: $\theta(\lambda r,\lambda s) = \lambda\theta(r,s)$ for all $r,s\in[0,\infty)$ and $\lambda\geq 0$;
\item behavior at $0$: $\theta(0,t)=0$ for all $t\in[0,\infty)$;
\end{itemize}
\item in addition there holds
\begin{itemize}
\item compatibility: $(\Psi^{*})'(\log s-\log t) \theta(s,t) = s-t$ for all $s,t>0$;
\item there exists a convex lower semi-continuous function $G_{\Psi^{*}}:[0,\infty)\times[0,\infty)\to[0,\infty)$ such that 
\[\frac14\Psi^{*}(\log t-\log s)\theta(s,t) =
  G_{\Psi^{*}}(s,t)\quad\forall s,t>0\;,\]
and such that $G_{\Psi^*}(s,t)=0$ if and only if $s=t$. 
\end{itemize}
\end{enumerate}
\end{assumption}

Let $\Psi:\R\to\R$ be the convex conjugate of $\Psi^{*}$ and note that it is strictly convex, strictly increasing, superlinear and even with $\Psi(0)=0$.

Given $\mu\in \cP(\R^{d})$, we define the measure $\nu_{\mu}\in\cM_{+}(\Omega)$ via 
\begin{equation}\label{eq:edge-measure}
\nu_{\mu}:=\theta\big(\frac{\dd\mu^{1}}{\dd\sigma},\frac{\dd\mu^{2}}{\dd\sigma}\big)\sigma\;,
\end{equation}
where $\mu^{1},\mu^{2}$ are given by \eqref{eq:defmu12} and $\sigma$ is any measure such that $\mu^{1},\mu^{2},\ll \sigma$. Due to the $1$-homogeneity of $\theta$ the definition is independent of $\sigma$. Note that if $\mu$ is absolutely continuous w.r.t.~Lebesgue measure with density $f$, then we have $\dd \nu_{\mu} = \theta(f \fs,\fp \fsp)B \dd v \dd \vs \dd \omega$.

We can now define the primal and dual dissipation potentials.

\begin{definition}\label{def:diss-pot} Given measures $\mu\in \cP(\R^{d})$, $\cU\in \cM(\Omega)$ we define
\begin{align}\label{eq:def-primal-diss}
\mathcal R(\mu,\cU):=\frac14 \int_{\Omega}\Psi\Big(\frac{\dd \cU}{\dd \nu_{\mu}}\Big)\dd\nu_{\mu}\;,
\end{align}
provided $\cU\ll \nu_{\mu}$ and $\mathcal R(\mu,\cU)=+\infty$ else. 
Given moreover a measurable function $\xi:\Omega\to\R$ we define
\begin{align}\label{eq:def-dual-diss}
 \mathcal R^{*}(\mu,\xi):=\frac14\int_{\Omega}\Psi^{*}\big(\xi\big)\dd\nu_{\mu}\;.
\end{align}
Finally, we define
\begin{equation}\label{eq:alt-dissipation}
D_{\Psi^{*}}(\mu):= \int_{\Omega} G_{\Psi^{*}}(f\fs,\fp\fsp)B \dd v\dd\vs\dd\omega \;,
\end{equation}
provided $\mu$ is absolutely continuous with density $f$ and set $D_{\Psi^{*}}(\mu)=+\infty$ otherwise.
\end{definition} 

The functional $D_{\Psi^{*}}$ takes over the role of the entropy dissipation is formally given by
\begin{equation*}
D_{\Psi^{*}}(\mu)= \mathcal R^{*}(\mu, -\bar\nabla \log f)\;,
\end{equation*}
provided $\mu$ has density $f$.

Note that the primal dissipation potential can be rewritten as the integral functional $\mathcal R(\mu,\cU)=\mathcal F_{\beta}(\mu^{1},\mu^{2},\cU)$ using the notation \eqref{eq:int-funct}, with the function $\beta$ defined by
\begin{equation*}
\beta(s,t,u):=\begin{cases}
\frac14\Psi\big(\frac{u}{\theta(s,t)}\big)\theta(s,t)\;, & \theta(s,t)\neq 0\;,\\
0\;, & \theta(s,t) = 0 \text{ and } u=0\;,\\
+\infty\;, \theta(s,t)=0 \text{ and } u\neq 0\;.
 \end{cases}
\end{equation*}
Since the $\beta:[0,\infty)\times[0,\infty)\times \R\to[0,\infty]$ is again convex and lower semi-continuous, in analogy to Lemma \ref{lem:lscaction} shows that $\mathcal R$ is convex and lower semicontinuous w.r.t. weak convergence of $\mu$ and weak$^{*}$ convergence of $\cU$. Similarly the assumptions on $G_{\Psi^{*}}$ guarantee that $D_{\Psi^{*}}$ is convex and lower semi-continuous w.r.t.~weak convergence.

We can now formulate the variational characterization of the Boltzmann equation.

\begin{theorem}\label{thm:ED-balance}
  Let $(\mu,\cU)\in\CE_T$ with $(\mu_t)_t\subset \mathcal P_p(\mathbb R^d)$ such that
 \begin{align}\label{eq:alt-f-not}
  \cH(\mu_0)<\infty\;, \qquad  \int_0^T\cE_{p}(\mu_t)\dd t<\infty\;. 
 \end{align}
 with $p=2+\max(\gamma,0)$. Then we have that:
 \begin{align}\label{eq:alt-chain-rule-est}
   \mathcal L_T(\mu,\cU):=\cH(\mu_T)-\cH(\mu_0)+\int_0^T D_{\Psi^{*}}(\mu_t) + \mathcal R(\mu_{t},\cU_{t})\dd t \geq 0 \;.
 \end{align}
 Moreover, we have $\mathcal L_T(\mu,\cU)=0$ if and only if $\mu_{t}$ has density $f_{t}$ with $(f_{t})_{t}$ a solution to
 the homogeneous Boltzmann equation.
\end{theorem}

\begin{proof}
We can follow with small modifications the proof of Proposition \ref{prop:chainrule} and Theorem \ref{thm:EDI}. Let us highlight the main steps.

First note that by the convex duality of $\Psi$ and $\Psi^{*}$, for $s,t>0$ we have the estimate
\begin{align}\nonumber
\frac14\left|(\log t-\log s)w\right| &= \frac14\left|(\log s-\log t)\frac{w}{\theta(s,t)}\right|\theta(s,t)\\\nonumber
 &\leq \frac14\Psi\big(\frac{w}{\theta(s,t)}\big)\theta(s,t) +\frac14 \Psi^{*}\big(\log s-\log t\big)\theta(s,t)\\\label{eq:dual-est}
&= \frac14\Psi\big(\frac{w}{\theta(s,t)}\big)\theta(s,t) + G_{\Psi^{*}}(s,t)\;.
\end{align}
Moreover, we have equality
\begin{align}
\label{eq:dual-eq}
\frac14(\log t-\log s)w &= -\frac14\Psi\big(\frac{w}{\theta(s,t)}\big)\theta(s,t) - G_{\Psi^{*}}(s,t)
\end{align}
 if and only if $w=(\Psi^{*})'\big(\log s-\log t\big)\theta(s,t)$ and hence by Assumption \ref{ass:gen-grad} if and only if $w=s-t$.

To prove \eqref{eq:alt-chain-rule-est} we can assume that $\int_{0}^{T}D_{\Psi^{*}}(\mu_t) + \mathcal R(\mu_{t},\cU_{t}) \dd t<\infty$ since otherwise the estimate holds trivially. Then arguing as in Lemma \ref{lem:densities} we have for a.e.~$t$ that $\mu_{t}$ and $\cU_{t}$ have densities $f_{t}$ and $U_{t}=W_{t}\theta(f_{t})B$, where $\theta(f):=\theta(f\fs,\fp,\fsp)$. In particular,  the set of $v,\vs,\omega$ where $\theta(f_{t})=0$ is negligible for $\cU_{t}$.

The first step is then to establish the chain rule
\begin{equation}\label{eq:chain-rule-alt}
 \cH(\mu_t)-\cH(\mu_s) = \int_s^t\frac14\int\dgrad\log f_r \dd\cU_r\dd r\quad\forall 0\leq s\leq t\leq T\;.
\end{equation}
According to the previous comment, the integral can be restricted to the set $\{\theta(f_{r})>0\}$ and is thus well-defined. One argues as in the proof of Proposition \ref{prop:chainrule} by regularization and replaces $\frac12\cA(\mu,\cU)$ and $\frac12 D(\mu)$ with $\mathcal R(\mu,\cU)$ and $D_{\Psi^{*}}(\mu)$. The estimate \eqref{eq:dual-est} yields the necessary majorants. The essential properties of $\cA$ and $D$ used in the proof were convexity and lower semi-continuity which still hold under our assumptions for $\mathcal R$ and $D_{\Psi^{*}}$. 

From \eqref{eq:chain-rule-alt} we obtain \eqref{eq:alt-chain-rule-est} by estimating as in \eqref{eq:dual-est} on the set $\{\theta(f)>0\}$:
\begin{align*}
\frac14\bar\nabla \log f \cdot U = -\frac14 \theta(f)(-\bar\nabla f)B\frac{U}{\theta(f)B}
\geq - \frac14 \Psi\big(\frac{U}{\theta(f)B}\big)\theta(f)B -\frac14\Psi^{*}\big(-\bar\nabla f\big)\theta(f)B\;.
\end{align*}
Now, assume that equality is attained in
\eqref{eq:alt-chain-rule-est}.  Then for a.e.~$t$ and a.e.~on the set
$\{\theta(f_t)>0\}$ we must have 
 \[U = (\Psi^{*})'(-\bar\nabla f)\theta(f)B=
   \big[f\fs-\fp\fsp\big]B\;.\] 
On the set where $\theta(f_t)=0$ and hence $U_t=0$ we must have
$G_{\Psi^*}(f\fs,\fp\fsp)=0$. But the assumptions on $G_{\Psi^*}$ and
$\theta$ then yield that on $\{\theta(f)=0\}$ we have $f\fs=\fp\fsp=0$
and thus again $U=f\fs-\fp\fsp$, i.e.~$(f_t)$ is a solution to the
Boltzmann equation. Conversely, a solution to the
Boltzmann equation leads to equality in \eqref{eq:alt-chain-rule-est}.
\end{proof}

\begin{example}\label{ex:quad-gen}
Finally let us highlight two examples of generalized gradient structures compatible with our Assumption \ref{ass:gen-grad}.
\begin{itemize}
 \item \emph{Quadratic gradient structure:} Choosing 
 \[\theta=\Lambda\;,\qquad  \Psi^{*}(r)=\Psi(r)=\frac12 r^{2}\;,\]
 with $\Lambda$ the logarithmic mean defined in \eqref{eq:log-mean}, we recover the framework considered in the previous section and the gradient flow characterization of Theorem \ref{thm:EDI}. Namely, $D_{\Psi^{*}}(\mu)=\frac12D(\mu)$ and $\mathcal R(\mu,\cU)=\frac12\cA(\mu,\cU)$ yields the action functional defined in Definition \ref{def:action}.
 \item \emph{$\cosh$ structure:} Let us set
 \[\theta(s,t) = \sqrt{st}\;,\qquad \Psi^{*}(\xi)=4\big(\cosh(\xi/2)-1\big)\;.\]
 Then we obtain 
 \[\Psi(s)=2s\log\Big(\frac{s+\sqrt{s^{2}+4}}{2}\Big) -\sqrt{s^{2}+4}+4\;,\]
 as well as
 \[G_{\Psi^{*}}(s,t) = \frac14\Psi^{*}\big(\log t-\log s\big)\theta(s,t)=\frac12\big(\sqrt{s}-\sqrt{t}\big)^{2}.\]
 \end{itemize}
 \end{example}
 Let us mention that generalized gradient structures involving $\cosh$ such as the latter example arise naturally in the context of large deviations for jump processes. Namely, the associated functional $\mathcal L_{T}$ from \eqref{eq:alt-chain-rule-est}  is the path-level large deviation rate functional for the empirical measure of $N$ independent copies of the process in the limit $N\to\infty$, see for instance \cite{MPR14, PRST20}.  In the present setting, the second structure in the above examples can formally be related with the large deviations of the Kac process considered in the next section. However, to the best of our knowledge no full large deviation principle for this classical Kac system with conservation of momentum and energy has been established yet. In \cite{BBBO21} a large deviations principle for a Kac type system with only conservation of momentum is established which gives rise to a $\cosh$-type generalized gradient structure for the corresponding limiting Boltzmann type equation.  A similar structure has also been used in \cite{BBB17} to give a variational characterization of linear spatially inhomogeneous Boltzmann equations.

\section{Consistency with Kac's random walk}
\label{sec:kac}

In this section we give a new proof of the convergence of Kac's random
walk to the solution of the spatially homogeneous Boltzmann equation,
see Theorem \ref{thm:kac-boltz-intro}, exploiting that both evolutions
have a gradient flow structure. We recall from
Section \ref{sec:kac-intro} that Kac's random walk is the continuous
time Markov chain on
\begin{align*}
  \cX_N := \left\{(v_1,\dots,v_N)\in\R^{dN} ~|~\sum_{i=1}^Nv_i=0\;,~\sum_{i=1}^N|v_i|^2=Nd  \right\}\;,
\end{align*}
with generator
\begin{align}\label{eq:Kac-generator}
  A f(\bv) = \frac{1}{N}\int_{S^{d-1}}\sum_{i<j}\left[f(R^\omega_{ij}\bv)-f(\bv)\right]B(v_i-v_j,\omega)\dd \omega\;,
\end{align}
where $R^\omega_{ij}\bv = (v_1,\dots,v_i',\dots,v_j',\dots,v_N)$, with
$v_i' = v_i - \ip{v_i-v_j,\omega}\omega$ and
$v_j' = v_j + \ip{v_i-v_j,\omega}\omega$. Let us denote by $\pi_N$ the normalized Hausdorff measure on
$\cX_N$ and note that the Markov chain is
reversible with respect to $\pi_N$. Denoting by $\mu^N_t$ the law of the chain starting in
$\mu_0^N$. Then its density $f^N_t$ w.r.t.~$\pi_N$ satisfies Kac's
master equation
\begin{align}\label{eq:kac-master}
\partial_t f^N_t=A f^N_t\;.
\end{align}
We recall the following result. For $\bv\in \R^{Nd}$ and $p\geq1$ we set
\begin{align*}
  \cE^N_p(\bv):=\frac1N\sum_{i=1}^N|v_i|^p\;.
\end{align*}
\begin{lemma}[Propagation of moments for Kac's random walk,~{\cite[Lem.~5.3]{MM13}}]\label{lem:kac-moments}
  Let $\mu_0^N$ an initial condition with $\ip{\cE^N_p,\mu^N_0}=\int \cE^N_p\dd\mu^N_0<\infty$. Then the law $(\mu^N_t)_{t\geq0}$ of Kac's random walk satisfies 
  \begin{align*}
    \sup_{t\geq0}\ip{\cE^N_p,\mu^N_t} \leq \max\{C_p,\ip{\cE^N_p,\mu^N_0}\}\;,
  \end{align*}
for some constant $C_p$ depending only on $p$.
\end{lemma}

We will first detail the gradient flow structure of the master
equation.

\subsection{Gradient flow structure}
\label{sec:kac-gf}

Kac's random walk possesses the structure of a gradient flow in
$\cP(\cX_N)$ of the relative entropy $\cH(\cdot|\pi_N)$ with respect
to a suitable geometry on $\cP(\cX_N)$ as we shall now describe. For
general Markov chains on finite state spaces a gradient flow structure
has been discovered in \cite{Ma11,Mie11}. Here we briefly show how to
extend this result to the present case of the continuous state space
$\cX_N$. The construction is similar as in Section \ref{sec:CRE-action},
see also \cite{Erb14}. Let us stress however that for the purpose of
showing consistency with the Boltzmann equation it will only be
important to know that the solution $(f_t)_t$ to \eqref{eq:kac-master}
satisfies the energy identity $J^N_T(f)= 0$, see \eqref{eq:kac-edi}
below.

We introduce a jump kernel on $\cX_N$ by setting
\begin{align*}
  J(\bv,\dd\bu)~=~ \frac{1}{2N}\int_{S^{d-1}}\sum_{i,j=1}^N\delta_{R^\omega_{ij}\bv}(\dd\bu)B(v_i-v_j,\omega)\dd\omega\;.
\end{align*}
Given a probability measure $\mu\in\cP(\cX_N)$ we define
$\mu^1,\mu^2\in\cM(\cX_N\times\cX_N)$ via
\begin{align}\label{eq:mu-aux}
  \dd\mu^1(\bv,\bu) = J(\bv,\dd\bu)\dd\mu(\bv)\;,\quad
  \dd\mu^2(\bv,\bu) = J(\bu,\dd\bv)\dd\mu(\bu)\;.
\end{align}
For a pair $(\mu,\cV)$ with $\mu\in\cP(\cX_N)$ and $\cV\in
\cM(\cX_N\times\cX_N)$ we define the action
\begin{align*}
  \cA^N(\mu,\cV) :=
  2\cF_\alpha(\mu^1,\mu^2,\cV)\;,
\end{align*}
where $\cF_\alpha$ is defined in \eqref{eq:F-gen}. We define a distance on $\cP(\cX_N)$ by
setting
\begin{align*}
  \cW_N(\mu_0,\mu_1)^2 := \inf_{\mu,\cV}\int_0^1\cA^N(\mu_t,\cV_t)\dd t\;, 
\end{align*}
where the infimum is taken over all curves $(\mu_t)_{t\in[0,1]}$
connecting $\mu_0$ to $\mu_1$ and all $(\cV_t)_{t\in[0,1]}$ subject to
the continuity equation
\begin{align*}
   \ddt \int_{\cX_N}\phi\dd\mu_t - \frac12\int_{\cX_N^2}[\phi(\bu)-\phi(\bv)]\dd\cV_t(\bv,\bu)= 0\;,\quad\forall \phi\in C_b(\cX_N)\;.
\end{align*}
It follows from the results in \cite[Thm.~4.4, Prop.~4.3]{Erb14}, by
considering $J$ as a jump kernel on the ambient space $\R^{dN}$, that
$\cW_N$ defines a distance and that the infimum in the definition is
attained by an optimal pair $(\mu,\cV)$. For a curve $(\mu_t)_{t\in[0,T]}$ in
$\cP(\cX_N)$ we define its action by
\begin{align*}
  \cA^N_{T}(\mu):=\inf\left\{\int_0^T\cA_N(\mu_t,\cV_t)\dd t\right\}\;,
\end{align*}
where the infimum is taken over all $(\cV_t)_t$ such that $(\mu,\cV)$
satisfy the continuity equation. There exists an optimal $\cV$
attaining the infimum, see \cite[Prop.~4.3]{Erb14}. In fact, for
a.e.~$t$, $\cA^N(\mu_t,\cV_t)$ equals the metric derivative of the
curve w.r.t.~$\cW_N$. We define the \emph{entropy dissipation} of
$\mu\in\cP(\cX_N)$ by
\begin{align*}
  D^N(\mu)=\frac{1}{4N}\int_{\cX_N}\int_{S^{d-1}}\sum_{i,j}&\Big[f(R^\omega_{ij}\bv)-f(\bv)\Big]\\
  \times&\Big[\log
  f(R^\omega_{ij}\bv)-\log f(\bv)\Big]B(v_i-v_j,\omega)\dd\omega\dd\pi_N(\bv)\;,
\end{align*}
provided $\mu=f\pi_N$ and we set $D^N(\mu)=+\infty$ if $\mu$ is not
absolutely continuous. Note that along any solution $f_t$ to the
master equation \eqref{eq:kac-master} we have
\begin{align}\label{eq:dissi-kac}
  \ddt\cH(f_t|\pi_N)=-D^N(f_t)\;.
\end{align}

\begin{proposition}\label{prop:kac-gf}
  For any curve $(\mu_t)_{t\in[0,T]}$ in $\cP(\cX_N)$ with
  $\cH(\mu_0|\pi_N)<\infty$ we have
  \begin{align}\label{eq:kac-edi}
  J_T^N(\mu)=\cH(\mu_T|\pi_N)-\cH_N(\mu_0|\pi_N)+\frac12\int_0^TD^N(\mu_t)\dd t +\frac12\cA^N_{T}(\mu) \geq 0\;.
  \end{align}
  Moreover, $J_T^N(\mu)=0$ holds if and only if $\mu_t=f_t\pi_N$ where $f_t$ solves \eqref{eq:kac-master}.
\end{proposition}

\begin{proof}
  We will focus on showing that any solution $(\mu_t)_t$ to the master
  equation \eqref{eq:kac-master} satisfies $J^N_T(\mu)=0$ since this
  will be used in the sequel. The other statements can be obtained by
  following a similar line of reasoning as in Section
  \ref{sec:gradflow}, namely establishing a chain rule for the entropy
  analogous to Proposition \ref{prop:chainrule} via a regularization
  argument (in fact the situation is much simpler due to linearity of
  the master equation).
 
  Let $\mu_t=f_t\pi^N$ be a solution to the master equation
  \eqref{eq:kac-master}. Then the couple $(\mu_t,\cV_t)$ solves
  the continuity equation if we choose
  $$\dd\cV_t(\bv,\bu)=\Psi_t(\bv,\bu)\Lambda(f_t(\bv),f_t(\bu))J(\bv,\dd\bu)\pi^N(\dd
  v)$$
  with $\Psi_t(\bv,\bu)=\log f_t(\bu)-\log f_t(\bv)$. Note moreover that
  $\cA(\mu_t,\cV_t)=D_N(\mu_t)$. Thus, integrating \eqref{eq:dissi-kac}
  yields $J_T(\mu)=0$.
\end{proof}

\subsection{Convergence to the Boltzmann equation}
\label{sec:convegence}

In this section we will give a new proof that the distribution of the
empirical measure of $N$ particles evolving by Kac's random walk
converges to the solution of the homogeneous Boltzmann equation as
$N\to\infty$. For convenience let us recall the setup and the
convergence statement.

Consider the map assigning to a configuration in $\cX_N$ its empirical
measure
\begin{align*}
  L_N: \cX_N\to \cP(\R^d)\;, \quad \bv \mapsto \frac{1}{N}\sum_{i=1}^N\delta_{v_i}\;.
\end{align*}
Let us set
\begin{align*}
  \cP_*(\R^d):=\big\{\mu\in\cP(\R^d): \cM(\mu)=0,\cE(\mu)=d\big\}\;,
\end{align*}
the set of probability measures with zero momentum and energy $d$,
recall \eqref{eq:momentum-energy}. Note that for any $\bv\in\cX_N$ we
have $L_N\bv\in \cP_*(\R^d)$. Let us denote by $M=M^{0,d}$ the
standard Maxwellian distribution and by $\cH(\mu|M)$ the relative
entropy, see \eqref{eq:ent-gauss}. We consider $\cP_*(\R^d)$ as a
subset of $\cP(\R^d)$ equipped with the topology of weak convergence.

\begin{theorem}\label{thm:kac-boltz}
  For each $N$ let $(\mu^N_t)_{t\geq0}$ be the law of Kac's random
  walk starting from $\mu^N_0$ and let $c^N_t:=(L_N)_\#\mu^N_t$ be the
  law of the empirical measures. Assume that $\mu^N_0$ is
  well-prepared for some $\nu_0=f_0\cL\in\cP_*(\R^d)$ (if $\gamma>0$ assume further $\cE_4(\nu_0)<\infty$) with
  $\cH(\nu_0|M)<\infty$ in the sense that in the limit $N\to\infty$
  \begin{align*}
    c^N_0\rightharpoonup \delta_{\nu_0}\;,\quad
    \frac1N\cH(\mu^N_0|\pi_N) \rightarrow \cH(\nu_0|M)\;.
  \end{align*}
 Assume further that for some $p>2+\max(\gamma,0)$  
  \begin{align*}
    \sup_N \ip{\cE^N_p,\mu^N_0} <\infty\;.
  \end{align*}
  Then, for all $t>0$, as $N\to\infty$ we have
   \begin{align}\label{eq:prop-chaos}
    c^N_t\rightharpoonup \delta_{\nu_t}\;,\quad \frac1N\cH(\mu^N_t|\pi_N) \rightarrow \cH(\nu_t|M)\;,
  \end{align}
 where $\nu_t=f_t\cL$ and $f_t$ is the unique
  solution to the spatially homogeneous Boltzmann equation with
  initial datum $f_0$.
\end{theorem}

The strategy of the proof will be to pass to the limit in the
variational formulation of the master equation and obtain the
variational formulation of the Boltzmann equation. The key ingredient
to this will be to establish $\liminf$ estimates relating the entropy,
dissipation and action for the Kac walk and the Boltzmann
equation. Although the proofs of the latter might seem long, the core
argument is rather simple and boils down to the lower semicontinuity
of integral functionals stated in Lemma \ref{lem:Fprops}. A
non-trivial additional ingredient that we develop is a probabilistic
representation result that allows to view certain curves in
$\cP(\cP_*(\R^d))$ as superposition of curves in $\cP_*(\R^d)$, see
Proposition \ref{prop:superposition}.

Let us now first give the proof of convergence theorem. Afterwards we
will develop the necessary ingredients.

\begin{proof}[Proof of Theorem \ref{thm:kac-boltz}]
  By Proposition \ref{prop:kac-gf} we have that $(\mu^N_t)_{t\geq0}$
  satisfies
\begin{align}\label{eq:kacboltz1}
 \cH(\mu^N_T|\pi_N)-\cH(\mu^N_0|\pi_N)+\frac12\int_0^TD^N(\mu^N_t)\dd t+\frac12\cA^N_T(\mu)= 0\;.
\end{align}
Together with the convergence of $\cH(\mu^N_0|\pi_N)/N$ this implies in particular 
\begin{align*}
  \sup\limits_N\frac1N\cA^N_T(\mu^N) < \infty\;.
\end{align*}
The compactness result Lemma \ref{lem:curve-conv} then yields that up
to a subsequence we have that $c^N_t\rightharpoonup c_t$ weakly for
all $t$ and a continuous curve $(c_t)_{t\geq0}$ in $\cP(\cP(\R^d))$ with $c_t$
concentrated on $\cP_*(\R^d)\cap \cPpE$ for all $t$ and suitable $E>0$. A priori, $c_t$ is not a
Dirac measure. However, by the superposition principle Proposition
\ref{prop:superposition} the curve $(c_t)_{t\in[0,T]}$ can be
represented as $c_t=(e_t)_\#\Theta$ for a probability measure $\Theta$
on $C\big([0,T],\cP(\R^d))$. Thanks to the $\liminf$-inequalities for
the entropy, dissipation and action given by \eqref{eq:entropy-conv},
\eqref{eq:dissipation-conv} and \eqref{eq:action-conv}, dividing by
$N$ in \eqref{eq:kacboltz1} and passing to the limes inferior we
obtain
\begin{align}\label{eq:kacboltz2}
  \int\left[\cH(\eta_T)-\cH(\eta_0)+\frac12\int_0^TD(\eta_t)\dd t+\frac12\cA_T(\eta)\right] \dd\Theta(\eta)\leq 0\;,
\end{align}
using also that $\cH(\eta|M)=\cH(\eta)+\cH(M)$ for
$\eta\in\cP_*(\R^d)$ and that $\eta_0,\eta_T\in\cP_*(\R^d)$ for
$\Theta$-a.e.~$\eta$. By Theorem \ref{thm:EDI} the integrand is
non-negative. Thus we have in fact equality in \eqref{eq:kacboltz2}
and we infer that $\Theta$ is concentrated on gradient flow curves
$(\eta_t)_t$, i.e.~satisfying $J_T(\eta)=0$. Since
$\Theta$-a.s.~$\eta_0=\nu_0$ and the unique gradient flow curve
starting from $\nu_0$ is given by $\nu_t=f_t\cL$ with $f_t$ the
solution to the Boltzmann equation with initial datum $f_0$. Thus, we
infer that $c_t=(e_t)_\#\Theta=\delta_{\nu_t}$ for all $t$ and that
the convergence of $c^N_t$ to $\delta_{\nu_t}$ holds for the full
sequence. Finally, we prove \eqref{eq:prop-chaos}. From the previous
discussion we retain that
\begin{align*}
  0 &\geq \liminf\limits_N\frac1N J^N_T(\mu^N) -J_T(\nu)= \liminf_N\frac1N \cH(\mu^N_T|\pi_N)- \cH(\nu_T|M)\\
&+\frac12 \left[\liminf_N\frac1N \int_0^TD_N(\mu^N_t)\dd t +\cA^N_T(\mu^N) -\int_0^TD(\nu_t)\dd t+\cA_T(\nu)\right] \geq 0\;.
\end{align*}
Using again \eqref{eq:entropy-conv}, \eqref{eq:action-conv},
\eqref{eq:dissipation-conv}, we infer that we have equality
\begin{align*}
  \liminf\limits_N\frac1N\cH(\mu^N_t|\pi_N)=\cH(\nu_t|M)\;.
\end{align*}
Since by the same argument this must hold for any subsequence, we
conclude the convergence \eqref{eq:prop-chaos} for the full sequence.
\end{proof}  

We now develop the ingredients to the previous proof. We will first
show that any sequence of curves in $\cP(\cX_N)$ with uniformly
bounded action after passing to the empirical measure admits a limit
curve in $\cP(\cP(\R^d))$. Then we will give a representation of this
curve as a superposition of curves in $\cP(\R^d)$ and establish
$\liminf$ inequalities for the action and dissipation of the limit
curve. Finally, we prove the $\liminf$ inequality for the entropy.

\subsubsection{Convergence to a limit curve}
\label{sec:curve-conv} 

\begin{lemma}\label{lem:curve-conv}
  Let $(\mu^N_t)_{t\in[0,T]}$ be a sequence of curves in $\cP(\cX_N)$
  such that
  \begin{align}\label{eq:action-bounded1}
    \sup\limits_N\frac1N\cA^N_T(\mu^N) < \infty\;,
  \end{align}
  and for some $p'>2+\max(\gamma,0)$
  \begin{align}\label{eq:moment-bounded}
    \sup_N\sup_{t\in[0,T]}\ip{\cE^N_p,\mu^N_t}<\infty\;.
  \end{align}
  Put $c^N_t=(L_N)_\#\mu^N_t$. Then there exists a continuous curve
  $(c_t)_{t\in[0,T]}$ in $\cP(\cP(\R^d))$ such that up to a
  subsequence we have that $c^N_t\rightharpoonup c_t$ weakly and $c_t$
  is concentrated on $\cP_*(\R^d)$ for all $t\in[0,T]$.
\end{lemma}

\begin{proof}
  We consider the set $\cPE$ of probability measures with energy less
  than $E$, with $E=d$, recall \eqref{eq:P2E}. Recall that $\cPE$ is
  compact w.r.t.~weak convergence, hence also $\cP(\cPE)$ is
  compact. On $\cPE$, weak convergence is equivalent to convergence of the first
  moment, or convergence in the $L^1$-Wasserstein
  distance $W_1$. Let us denote by $\widetilde W_1$ the
  $L^1$-Wasserstein distance on $\cP(\cPE)$ induced by the
  $L^1$-Wasserstein distance $W_1$ on $\cPE$. Since $W_1$ is bounded on $\cPE$, $(\cP(\cPE),\widetilde W_1)$ is compact. We claim that 
  \begin{align}\label{eq:W-W1}
    \cW_N(\mu^N_s,\mu^N_t) \geq \frac{C}{\sqrt{N}}W_{1,d}(\mu^N_s,\mu^N_t) \geq C\sqrt{N}\widetilde W_1(c^N_s,c^N_t)\;,
  \end{align}
  for some universal constant $C>0$. Indeed, to prove the first
  inequality we view $\mu^N_s,\mu^N_t$ as measures on $\R^{Nd}$
  equipped with the distance $d(\bv,\bu)=\sum_i|v_i-u_i|$ and let
  $W_{1,d}$ denote the $L^1$-Wasserstein distance induced by $d$. Then
  the estimate follows from \cite[Prop.~4.5]{Erb14} once we note that
  $\int_{\R^{Nd}}d(\bv,\bu)^2J(\bv,\dd \bu)\leq CN$, where $C$ depends
  on the moment bound in \eqref{eq:moment-bounded}. The second
  inequality follows from the fact that the map $L_N$ is
  $1/N$-Lipschitz from $(\R^{Nd},d)$ to $(\cP(\R^d),W_1)$. Together
  with \eqref{eq:W-W1}, \eqref{eq:action-bounded1} implies that the
  curves $(c^N_t)_t$ are uniformly equicontinuous in $\cP(\cPE)$
  w.r.t.~the distance $\widetilde W_1$. Thus, the Arzela--Ascoli
  theorem yields that there exists a continuous curve
  $(c_t)_{t\in[0,T]}$ in $\cP(\cPE)$ such that up to extraction of a
  subsequence we have that $c^N_t\rightharpoonup c_t$ weakly for all
  $t\in[0,T]$.

  Finally, assume in addition \eqref{eq:moment-bounded} and let us
  show that $c_t$ is concentrated on $\cP_*(\R^d)$ for all $t$. We
  need to show $c_t\big(\{\cM=0,\cE=d\}\big)=1$. Since
  $c^N_t\big(\{\cM=0\}\big)=1$ and $\cM$ is continuous on $\cPE$, and
  hence $\{\cM=0\}$ is closed, the weak convergence
  $c^N_t\rightharpoonup c_t$ implies that
  $c_t\big(\{\cM=0\}\big)=1$. It remains to show that
  $c_t\big(\{\cE=d\}\big)=1$. Since $c_t$ is concentrated on
  $\cPE=\{\cE\leq d\}$, it suffices to show that
  $\ip{\cE,c_t}=\lim_N\ip{\cE,c^N_t}=d$. Set
  $\cE_p(\eta):=\int |v|^p\dd\eta(v)$, then \eqref{eq:moment-bounded}
  implies that for any $t$:
  \begin{align}\label{eq:moment-bound2}
    \sup_N \ip{\cE_p,c^N_t}<\infty\;.
  \end{align}
  Note that $\cE_2=\cE$. Since by Jensen's inequality we have
  $\cE_2(\nu)^{p/2}\leq \cE_p(\nu)$, \eqref{eq:moment-bound2} readily
  yields that $\cE_2$ is uniformly integrable
  w.r.t.~$c^N_t$. Moreover,
  $\sup_Nc^N_t\big(\{\cE_{2+\eps}\geq R\}\big)\to0$ as $R\to\infty$
  for $\eps<p-2$ and $\cE_2$ is continuous on
  $\{\cE_{2+\eps}\leq R\}$. Thus the we obtain the desired convergence
  $\ip{\cE_2,c_t}=\lim_N\ip{\cE_2,c^N_t}$, see
  e.g.~\cite[Prop.~5.1.10]{AGS08}.
\end{proof}

\subsubsection{Superposition principle and limits for the action and dissipation}
\label{sec:comp-liminf-action-dissipation}

\begin{proposition}[Superposition principle for the limit curve]\label{prop:superposition}
  Let $(\mu^N_t)_{t\in[0,T]}$ be a sequence of curves in $\cP(\cX_N)$
  satisfying \eqref{eq:action-bounded1} and \eqref{eq:moment-bounded},
  put $c^N_t=(L_N)_\#\mu^N_t$, and let $(c_t)_{t\in[0,T]}$ be the
  limit curve of Lemma \ref{lem:curve-conv}. Then, there exists a Borel
  probability measure $\Theta$ on $C\big([0,T],\cP(\R^d)\big)$ and a Borel
  family of measures $(\cU^\eta_t)_{t\in[0,T],\eta\in\cP(\R^d)}$ such that
  the following hold:
  \begin{itemize}
  \item $c_t=(e_t)_\#\Theta$ for all $t\in[0,T]$,
  \item for $\Theta$-a.e.~curve $(\eta_t)_{t\in[0,T]}$, the pair
    $(\eta_t,\cU^{\eta_t}_{t})_{t\in[0,T]}$ belongs to $\CE_T$, and $\eta_t\in \cPpE$ for $p=2+\max(\gamma,0)$ and suitable $E>0$ and all $t\in [0,T]$.   
  \end{itemize}
\end{proposition}

\begin{proposition}[$\liminf$-inequality for action and
  dissipation]\label{prop:liminf-action-diss}
   In the setting of Proposition \ref{prop:superposition} we have
  \begin{align}
    \label{eq:action-conv}
    \liminf\limits_N \frac1N\cA^N_T(\mu^N) &\geq \int\cA_T(\eta)\;\dd\Theta(\eta)\;,\\    \label{eq:dissipation-conv}
    \liminf\limits_N \frac1N \int_0^TD^N(\mu^N_t)\dd t &\geq \int\left[\int_0^T D(\eta_t)\dd t\right] \dd \Theta(\eta)\;,
  \end{align}
  where $\cA_T(\eta)$, $D(\eta)$ are the action and dissipation defined in
  \eqref{eq:action-curve-def},
 \eqref{eq:def-dissipation}.
\end{proposition}

In order to prove the superposition principle Proposition
\ref{prop:superposition}, we will describe curves in $\cP(\cP(\R^d)$
as curves in $\cP(\R^\infty)$ by choosing a countable number of
coordinates given by integrals against test functions. This allows to
employ a superposition principle for solutions to the continuity
equation over $\R^\infty$ by Ambrosio and Trevisan \cite{AT14}. Let us
briefly recall this result.

Consider $\R^\infty=\R^\N$ and let $p_i:\R^\infty\to\R$ be the natural
projections for $i\in\N$ and let
$\pi_n=(p_1,\dots,p_n):\R^\infty\to\R^n$. Equip $\R^\infty$ with the separable and complete distance
\begin{align*}
  d_\infty(x, y)=\sum_{n=1}^\infty 2^{-n}\min\{1,|p_n(x)-p_n( y)|\}\;.
\end{align*}
In a similar way, $C\big([0,T],\R^\infty\big)$ can be equipped with a
separable and complete distance. We denote by
$AC_w\big([0,T],\R^\infty\big)$ the subset of
$C\big([0,T],\R^\infty\big)$ consisting of all $\gamma$ such that
$p_i\circ\gamma\in AC\big([0,T],\R\big)$ for all $i$.

A function $F:\R^\infty\to \R$ is called \emph{smooth cylindrical}, if it is of the form
\begin{align*}
  F(x)=\psi\big(p_1(x),\dots,p_n(x)\big)\;,
\end{align*}
for some $\psi\in C^1_b(\R^n)$ and $n\in\N$. Its gradient $\nabla F:\R^\infty\to\R^\infty$ is defined by
\begin{align*}
  \nabla F(x)=\left(\partial_1\psi\big(\pi_n(x)\big),\dots,\partial_n\psi\big(\pi_n(x)\big),0,0,\dots\right)\;.
\end{align*}
Then, we have the following representation result.
\begin{theorem}{\cite[Thm.~7.1]{AT14}}\label{thm:rep}
  Let $\bb:(0,T)\times\R^\infty\to\R^\infty$ be a Borel vector field
  and let $(\nu_t)_{t\in(0,T)}$ be a family of Borel probability
  measures on $\R^\infty$ continuous in duality with smooth cylinder
  functions satisfying
  \begin{align}\label{eq:cond1}
    \int_0^T|p_i(\bb_t)|\dd\nu_t\dd t <\infty \quad \forall i\in \N\;,
  \end{align}
  and in the sense of distributions in $(0,T)$
  \begin{align}\label{eq:cond2}
    \ddt \int F\dd\nu_t = \int (\bb_t,\nabla F)\dd\nu_t\quad\forall F \text{ smooth cylindrical}\;.
  \end{align}
  Then, there exists a Borel probability measure $\boldsymbol \lambda$
  on $C\big([0,T],\R^\infty\big)$ satisfying
  $(e_t)_\#\boldsymbol\lambda=\nu_t$ for all $t$, concentrated on
  $\gamma\in AC_w\big([0,T],\R^\infty\big)$ solving the ODE
  $\dot\gamma=\bb_t(\gamma)$ a.e.~in $(0,T)$.
\end{theorem}

\begin{proof}[Proof of Proposition \ref{prop:superposition}]
  We will proceed in 3 steps. Starting from a solution to the discrete
  continuity equation over $\cX_N$ we pass to the empirical measure
  and obtain a limiting family of collision rates $\cU^\eta_t$. Then, by choosing integrals against a
  collection of test functions as coordinates, we describe the
  limiting curve $c$ via a continuity equation over $\R^\infty$ with a
  vector field determined by the collision rates
  $\cU^\eta_t$. Finally, we apply the superposition principle for
  $\R^\infty$ and see that the obtained random curve in $\R^\infty$ is
  indeed the coordinate description of a random curve $(\eta_t)$ in $\cP(\R^d)$
  solving the collision rate equation driven by the rates $\cU^{\eta_t}_t$.

  \emph{Step 1: Limiting collision rate.} Recall from Section
  \ref{sec:kac-gf} that we can choose measures
  $\cV^N_t\in \cM(\cX_N\times\cX_N)$ such that
  $\cA^N_T(\mu^N)=\int_0^T\cA_N(\mu^N_t,\cV^N_t)\dd t$. Let us define
  the measures $\cV^N:=\cV^N_t\dd t$ and $\mu^{N,k}:=\mu^{N,k}_t\dd t$, $k=1,2$, 
  in $\cM(\cX^N\times\cX^N\times[0,T])$.  Note that by the structure of
  the jump kernel $J$, for any~$(\bv,\bu)$ in the support of
  $\mu^{N,1}_t,\mu^{N,2}_t$, with $\bv\neq \bu$, there exist
  unique $(i,j,\omega)$ with $1\leq i<j\leq N$, $\omega\in S^{d-1}$
  such that $\bu=R^\omega_{ij}(\bv)$ (when $\bv=\bu$, we pick $i=j$
  and $\omega$ at random). We push forward $\cV^N,\mu^{N,k}$ by the
  map $(\bv,\bu)\mapsto (L_N(\bv),L_N(\bu),v_i,v_j,\omega)$ with
  $i,j,\omega$ as above. This defines measures $\gamma^N, \beta^{N,k}$
  on $\cPpE^2\times (\R^d)^2\times S^{d-1}\times[0,T]$. We find that
  \begin{align}\nonumber
    \dd\beta^{N,1}(\eta,\eta',v,\vs,\omega,t)&=\frac{N}{2}\delta_{\eta^{N,v,\vs,\omega}}(\dd\eta')B(v-\vs,\omega)\eta(\dd v)\eta(\dd\vs)\dd\omega \dd c^N_t(\eta)\dd t\\\label{eq:betaN1}
                                             &=\frac{N}{2}\delta_{\eta^{N,v,\vs,\omega}}(\dd\eta')\dd\eta^1(v,\vs,\omega)\dd c^N_t(\eta)\dd t\;,\\\nonumber
    \dd\beta^{N,2}(\eta,\eta',v,\vs,\omega,t)&=
\frac{N}{2}\delta_{\eta'^{N,T_\omega^{-1}(v,\vs),\omega}}(\dd\eta)B(v-\vs,\omega)\\\nonumber
&\qquad\dd(T_\omega)_\#\eta'^{\otimes2}(v,\vs)\dd\omega \dd c^N_t(\eta')\dd t\\\label{eq:betaN2}
&=      \frac{N}{2}\delta_{\eta'^{N,T_\omega^{-1}(v,\vs),\omega}}(\dd\eta)\dd\eta'^2(v,\vs,\omega)\dd c^N_t(\eta')\dd t\;,
  \end{align}
  where we set
  $\eta^{N,v,\vs,\omega}=\eta+\frac1N(\delta_{\vp}+\delta_{\vsp}-\delta_v-\delta_{\vs})$
  with $v,\vs,\vp,\vsp$ related via \eqref{eq:pre-post2} and recall
  that $c^N_t=(L_N)_\#\mu^N_t$ and recall the notation \eqref{eq:defmu12}. To see this, note that $L_N(\bu)=L_N(\bv)^{N,v_i,v_j,\omega}$ if $\bu=R^\omega_{ij}(\bv)$ and that we can write
\begin{align*}
  \sum_{i,j=1}^Nf(v_i,v_j)
  = N^2\int f(v,\vs) L_N(\bv)(\dd v) L_N(\bv)(\dd \vs)\;.  
\end{align*}
To obtain the expression for $\beta^{N,2}$, note further, that if
$\bv=R^\omega_{i,j}(\bu)$, we have that $(v_i,v_j)=T_\omega(u_i,u_j)$.
 
From the weak convergence of $c^N_t$ to $c_t$ for all $t$ granted by Lemma \ref{lem:curve-conv}, we infer that as $N\to\infty$ we have
$\frac2N\beta^{N,k}\rightharpoonup\beta^k$ in duality with $C_b$ where
\begin{align}\label{eq:betak}
  \dd\beta^k(\eta,\eta',v,\vs,\omega,t)=\delta_\eta(\dd\eta')\dd\eta^k(v,\vs,\omega)\dd c_t(\eta)\dd t\;.
\end{align}
From Lemma \ref{lem:Fprops} (ii) we infer that
\begin{align*}\
  \cF_\alpha\left(\frac2N\beta^{N,1},\frac2N\beta^{N,2},\frac2N\gamma^N\right)\leq \frac2N\cF_\alpha\big(\mu^{N,1},\mu^{N,2},\cV^N\big)=\frac1{N}\cA^N_T(\mu^N)\;,
\end{align*}
and the last expression is bounded by assumption. From Lemma
\ref{lem:integrability} we infer as in the proof of Proposition
\ref{prop:opt-collision-rate} that $\frac2N\gamma^N$ has uniformly bounded
variation and hence converges weakly* up to a further subsequence to a limit
$\gamma$. By lower semicontinuity and homogeneity of $\cF_\alpha$
we find
\begin{align}\label{eq:lsc-act-1}
    \cF_\alpha(\beta^1,\beta^2,\gamma)\leq \liminf\limits_N\frac1N\cA^N_T(\mu^N)\;.
 \end{align}
 As in Lemma \ref{lem:densities} we infer from finiteness of the left hand side that
 $\gamma$ is absolutely continuous w.r.t.~the measure 
$L:=\delta_\eta(\dd\eta')\Lambda(\eta^1,\eta^2)c_t(\dd\eta)\dd t$, where $\Lambda(\eta^1,\eta^2):=\Lambda(\frac{\dd\eta^1}{\dd\sigma},\frac{\dd\eta^2}{\dd\sigma})\dd\sigma$ for any $\sigma$ such that $\eta^1,\eta^2\ll\sigma$. Hence there exists a Borel function $U:\cPpE^2\times(\R^d)^2\times S^{d-1}\times[0,T]\to \R$ such that $\gamma=UL$ and we can write 
 \begin{align}\label{eq:gamma-rep}
   \dd\gamma(\eta,\eta',v,\vs,\omega,t)=\delta_\eta(\dd\eta')\dd \cU^\eta_{t}(v,\vs,\omega)\dd c_t(\eta)\dd t\;,
 \end{align}
 where $(\cU^\eta_{t})_{\eta,t}$ is the Borel family of measures defined by $$\dd \cU^\eta_t(v,\vs,\omega)=U(\eta,\eta,v,\vs,\omega,t)\dd\Lambda(\eta^1,\eta^2)(v,\vs,\omega)\;.$$ Note further that
 \begin{align}\label{eq:action-conv2}
    \cF_\alpha(\gamma,\beta^1,\beta^2)= \int_0^T\int\cA(\eta,\cU^\eta_{t})\dd c_t(\eta)\dd t\;.
 \end{align}
\smallskip

\emph{Step 2: Continuity equation in $\R^\infty$.}  We now describe the
curve $(c_t)$ as an evolution in $\cP(\R^\infty)$. Fix a countable
collection $\{f_i\}_{i\in\N}$ of functions that is dense
(w.r.t.~uniform convergence) in the set of $1$-Lipschitz functions on
$\R^d$ vanishing at $0$. Define a map $I:\cPpE\to\R^\infty$ by setting
 \begin{align*}
   I(\eta):= \big(\ip{f_1,\eta},\ip{f_2,\eta},\dots\big)\;,
 \end{align*}
 and write $I^m=\pi_m\circ I$. Note that $I$ is injective and continuous w.r.t.~the distance $W_1$
 on $\cPpE$ by Kantorovich duality. $I(\cPpE)$ is closed in $\R^\infty$,
 since $(\cPpE,W_1)$ is compact, and $I^{-1}:I(X)\to\cPpE$ is
 continuous w.r.t.~$W_1$. 

 We define a curve $(\nu_t)_{t\in[0,T]}$ via $\nu_t:=I_\#c_t$ and note
 that it is continuous in duality with smooth cylinder functions by continuity of $t\mapsto c_t$. We define a Borel vector
 field $\bb:(0,T)\times\R^\infty\to\R^\infty$ via
 \begin{align}\label{eq:def-bb}
   \bb^i_t(x)=
   \begin{cases}
   \frac14\int \bar\nabla f_i\dd\cU^\eta_t & x=I(\eta)\in I(\cPpE)\;,\\
    0\;, &x\notin I(\cPpE)\;.
   \end{cases}
 \end{align}
 We claim that $(\nu,\bb)$ satisfies the continuity equation in
 $\R^\infty$, i.e.~\eqref{eq:cond1}, \eqref{eq:cond2}. Indeed,
 \eqref{eq:cond1} follows from \eqref{eq:action-conv2} and
 \eqref{eq:lsc-act-1} with Corollary \ref{cor:more-integrability}. To show \eqref{eq:cond2}, fix a smooth cylinder
 function $F(\bx)=\psi\big(p_1(x),\dots,p_n(x)\big)$ and $a\in
 C^\infty_c(0,T)$. From the continuity equation for $(\mu^N_t,\cV^N_t)$ we obtain after
 passing to the empirical measure
\begin{align*}
  \int_0^Ta'(t)\int F\circ I\dd c^N_t\dd t &= -\frac{1}{2}\int a(t)\big[F\big(I(\eta^{N,v,\vs,\omega})\big)-F\big(I(\eta)\big)\big]\dd\gamma^N\;.
\end{align*}
Note that
$F\big(I(\eta^{N,v,\vs,\omega})\big)-F\big(I(\eta)\big)=\frac1N\sum_i\partial_i\psi\big(I^m(\eta))\bar\nabla
f_i(v,\vs,\omega)+o(1)$.
We infer from the convergence of $c^N_t$ to $c_t$ and of
$\frac2N\gamma^N$ to $\gamma$ and \eqref{eq:gamma-rep} that
\begin{align}\nonumber
  \int_0^T a'(t) \int F\circ I \dd c_t~\dd t &= -\frac14 \int_0^Ta(t)\int  \sum_i\partial_i\psi\big(I^m(\eta)\big) \bar\nabla f_i\dd\cU^{\eta}_t\dd c_t(\eta)\dd t\\\label{eq:limit-ce1}
                                             &=-\int_0^Ta(t)\int \ip{\bb_t,\nabla F}\dd\nu_t\dd t \;,
\end{align}
which is \eqref{eq:cond2}.  \smallskip

\emph{Step 3: Probabilistic representation.} By Theorem \ref{thm:rep}
there exists a Borel probability measure $\boldsymbol\lambda$ on
$C\big([0,T],\R^\infty\big)$ concentrated on the solutions
$\gamma\in AC_w\big([0,T],\R^\infty\big)$ to the ODE
$\dot\gamma=\bb_t(\gamma)$ such that
$(e_t)_\#\boldsymbol\lambda=\nu_t$ for all $t$. Since $\nu_t$ is
concentrated on the closed set $I(\cPpE)$ for all $t$ we have that
$\bx_t\in I(\cPpE)$ for all $t\in[0,T]$ and $\boldsymbol\lambda$
a.e.~$\gamma$. Thus we can set $\Theta=\iota_{\#}\boldsymbol \lambda$,
where $\iota$ maps $\gamma\in C\big([0,T],\R^\infty\big)$ to
$I^{-1}\circ \gamma\in C\big([0,T],\cPpE\big)$.  It remains to check
that $\Theta$ has the desired properties.

   Since $\nu_t=I_\#c_t$ we immediately get $(e_t)_\#\Theta=c_t$ for all $t$.
   Further, since for fixed $i$ we have $\ip{f_i,\iota(\gamma)}=\pi_i(\gamma)$, we have by \eqref{eq:def-bb} that
   $t\mapsto \ip{f_i,\eta_t}$ is absolutely continuous and 
   \begin{align}\label{eq:CREi}
     \ddt \ip{f_i,\eta_t} = +\frac14\int \bar\nabla f_i\dd\cU^{\eta_t}_t\quad \text{for a.e.~$t\in(0,T)$, for $\Theta$-a.e.~$\eta$}\;. 
   \end{align}
   From \eqref{eq:action-conv2} and \eqref{eq:lsc-act-1} we obtain
   with Corollary \ref{cor:more-integrability} that the integrability condition
   \eqref{eq:more-integrability-U} holds (with $p=1$). This allows to extend
   \eqref{eq:CREi} to all Lipschitz $f$. Hence for $\Theta$-a.e.~curve
   $\eta$ we have that $t\mapsto(\eta_t,\cU^{\eta_t}_t)$ belongs to
   $\CE_T^E$.
\end{proof}

\begin{proof}[Proof of Proposition \ref{prop:liminf-action-diss}:]
 We recall from \eqref{eq:lsc-act-1} and \eqref{eq:action-conv2} that 
 \begin{align}\label{eq:liminf-action}
   \int_0^T\int\cA(\eta,\cU^\eta_{t})\dd c_t(\eta)\dd t
   \leq
   \liminf\limits_N\frac1N\cA^N_T(\mu^N)\;.
 \end{align}
 We obtain a $\liminf$ estimate for the dissipation in a similar
 fashion. We note that $D^N(\mu^N_t)=2\cG(\mu^{N,1},\mu^{N,2})$, where
 $\cG$ is the integral functional defined in the proof of Lemma
 \ref{lem:diss-lsc}. From Lemma \ref{lem:Fprops} we obtain
  \begin{align}\nonumber
     \liminf\limits_N \int_0^T\frac1{N}D^N(\mu^N_t)\dd t &\geq \liminf\limits_N\cG\left(\frac2N\beta^{N,1},\frac2N\beta^{N,2}\right)
    \geq\cG\left(\beta^{1},\beta^{2}\right)\\\label{eq:dissipation-conv2} &= \int_0^T\int D(\eta)\dd c_t(\eta)\dd t\;,
   \end{align}
   where we recall the definition of $\beta^{N,k}$ and $\beta^k$ from
   \eqref{eq:betaN1}, \eqref{eq:betaN2}, \eqref{eq:betak}.  By
   Proposition \ref{prop:superposition} we can then
   rewrite \eqref{eq:liminf-action} and \eqref{eq:dissipation-conv2} as
   \eqref{eq:action-conv} and \eqref{eq:dissipation-conv}, noting that
   $\Theta$-a.e.~curve $(\eta_t)_t$ satisfies $\cA_T(\eta)\leq\int_0^T\cA(\eta_t,\cU_t^{\eta_t})\dd t$.
\end{proof}

\subsubsection{Limit for the relative entropy}
\label{sec:liminf-ent}


\begin{proposition}[$\liminf$-inequality for the entropy]\label{prop:entropy-conv}
  Let $(\mu^N)_N$ be a sequence of measures in $\cP(\cX_N)$ such that
  $c^N=(L_N)_\#\mu^N$ converges weakly to $c\in\cP(\cPpE)$. Then we have that
  \begin{align}\label{eq:entropy-conv}
    \liminf\limits_N \frac1N \cH(\mu^N|\pi_N) &\geq \int\cH(\eta|M)~\dd c(\eta)\;.
  \end{align}
\end{proposition}

To prove this result, we will rely on ideas from large deviation
theory. Namely, we will exploit the fact that the empirical measure of
independent Gaussian distributed points in $\R^d$ satisfies a large
deviation principle and that this implies a $\Gamma-\liminf$
inequality for the relative entropy w.r.t.~the law of this empirical
measure. Then we will conclude by relating the entropy w.r.t.~$\pi_N$
to the entropy w.r.t.~the product Gaussian distribution.  Let us
briefly explain the concepts we will be using. For background on large
deviation theory we refer to \cite{Dembo1998}.

Let $\mathcal X$ be a Polish space and equip the set of Borel probability
measures $\cP(\mathcal X)$ with the weak topology. Let $I:\mathcal X\to [0,\infty]$ be a
lower semicontinuous function. A sequence of measures $(m_N)_N$ in
$\cP(\mathcal X)$ is said to satisfy a \emph{large deviation principle} with
\emph{rate function} $I$ (and speed $N$) if for any open set $O$ and
any closed set $C$ in $\mathcal X$ their probabilities are asymptotically
controlled as:
  \begin{align*}
  \liminf _N\frac1N\log m_N(O) \geq -\inf_{x\in O}I(x)\;,\quad
 \limsup _N\frac1N\log m_N(C)   \leq -\inf_{x\in C}I(x)\;.
 \end{align*}
 If the second inequality holds only for all compact sets $C$, we speak
 of a \emph{weak large deviation upper bound}. This weak upper bound
 is equivalent to a $\Gamma-\liminf$ inequality for the
 relative entropy: 
 \begin{lemma}[{\cite[Thm.~3.5]{Mar18} (P1)$\Leftrightarrow$(H2)}]\label{lem:ldp-gamma}
  $(m_N)$ satisfies a weak large deviation upper bound with rate function $I$ and speed $N$ if and only if for any sequence $(\mu_N)$ in $\cP(\mathcal X)$ converging to $\mu$ we have
  \begin{align*}
    \liminf_N\frac1N \cH(\mu_N|m_N) \geq \int_{\mathcal X} I\dd\mu\;.
  \end{align*}
 \end{lemma}

 We will also use the following desintegration principle for the
 relative entropy, which can be verified by a direct computation.
 
 Let $\mathcal Y$ be a further Polish space, $\mu, m$ two probability measures
 on $\mathcal X$, and $T:\mathcal X\to \mathcal Y$ be a Borel map. Let $\mu(\cdot|T=y)$ and
 $m(\cdot|T=y)$ denote the desintegration of $\mu$ and $m$
 w.r.t.~$T$. I.e.~$\mu(\cdot|T=y)$ are probability measures
 concentrated on $T^{-1}(y)$ such that for any measurable set
 $A\subset\cX$, $y\mapsto\mu(A|T=y)$ is measurable, and
\begin{align*}
  \mu(A)= \int_{\mathcal Y}\mu(A|T=y)\dd T_{\#}\mu(y)\;,
\end{align*}
and similarly for $m$.
Then we have that
\begin{align}\label{eq:ent-desint}
  \cH(\mu|m)= \cH(T_{\#}\mu|T_{\#}m) +\int_{\mathcal Y} \cH\big(\mu(\cdot|T=y)|m(\cdot|T=y)\big)\dd T_{\#}\mu(y)\;.
\end{align}
Since the relative entropy is non-negative, we have in particular
\begin{align}\label{eq:ent-push}
  \cH(\mu|m)\geq \cH(T_{\#}\mu|T_{\#}m)\;.
\end{align}

\begin{proof}[Proof of Proposition \ref{prop:entropy-conv}:]
  (i) Let $\gamma_N\in\cP(\R^{Nd})$ denote the distribution of $N$ independent standard $d$-dimensional Gaussian vectors, i.e. $\gamma_N$ has density 
  \begin{align*}
   g_N(v_1,\dots,v_N) = (2\pi)^{-Nd/2}\exp\left(-\sum_{i=1}^N\frac{|v_i|^2}{2}\right)\;
  \end{align*}
w.r.t.~Lebesgue measure on $\R^{Nd}$. Note that $\pi_N$ is obtained by conditioning $\gamma_N$ to $\cX_N\subset\R^{dN}$, i.e.
\begin{align*}
  \pi_N=\gamma_N\left(\cdot~|~\cM^N=0,\cE^N=d\right) = \frac{g_N}{\int_{\cX_N}g_N\dd\pi_N}\pi_N\;,
\end{align*}
 with $\cM^N(\bv)=1/N\sum_iv_i$ and $\cE^N(\bv)=1/N\sum_i|v_i|^2$. This follows immediately from $g_N$ being constant on $\cX_N$.

  (ii) We now claim that the analog of \eqref{eq:entropy-conv} holds for $\gamma_N$: if $\tilde\mu^N$ is a sequence in $\cP(\R^{Nd})$ such that $c^N=(L_N)_\#\tilde\mu^N$ converges weakly to $c$, then
  \begin{align}\label{eq:entropy-conv-gamma}
    \liminf\limits_N \frac1N \cH(\tilde\mu^N|\gamma_N) &\geq \int\cH(\eta|M)~\dd c(\eta)\;.
  \end{align} 
  Setting $m_N:=(L_N)_\#\gamma_N$ we obtain from
  \eqref{eq:ent-push} that
  $\cH(\tilde \mu^N|\gamma_N)\geq \cH(c^N|m_N)$. Thus, it suffices to
  show that
  \begin{align}\label{eq:ent-conv-aux}
    \liminf_N\frac1N\cH(c^N|m_N)\geq \int \cH(\eta|M)c(\dd\eta)\;.
  \end{align}
  By Sanov's theorem on large deviations for empirical measures
  \cite[Thm.~6.2.10]{Dembo1998}, $m_N$ satisfies a large deviation
  principle with rate function $\cH(\cdot|M)$ on $\cP(\R^d)$ equipped
  with the weak topology. Thus, \eqref{eq:ent-conv-aux} follows from
  Lemma \ref{lem:ldp-gamma}.

  (iii) Finally, we will conclude by relating $\cH(\cdot|\gamma_N)$
  and $\cH(\cdot|\pi_N)$. For $m\in\R^d$, $E>0$, define
  $\Psi_{m,E}:\R^{Nd}\to\R^{Nd}$ by
  $\Psi_{m,E}(\bv)=(\sqrt{E}v_1+m,\dots,\sqrt{E}v_n+m)$. Let
  $Q_N=(\cM^N,\cE^N)_\#\gamma_N$ in $\cP\big(\R^d\times[0,\infty)\big)$ be
  the distribution of momentum and energy under $\gamma_N$. We have
  $\gamma_N(\cdot~|~\cM^N=m,\cE^N=E)=(\Psi_{m,E/d})_\#\pi_N$ as in
  (i). Hence $\gamma_N$ disintegrates as
  $\gamma_N=\int(\Psi_{m,E/d})_\#\pi_N\dd Q_N(m,E)$. Define a map
  $\boldsymbol\Psi:\cP(\cX_N)\to\cP(\R^{Nd})$ via
\begin{align*}
  \boldsymbol\Psi(\mu)=\int(\Psi_{m,E/d})_\#\mu~\dd Q_N(m,E)\;.
\end{align*}
Note that $(\cM^N,\cE^N)_\#\boldsymbol\Psi(\mu)=Q_N$. Thus, the desintegration formula \eqref{eq:ent-desint} with $T=(\cM^N,\cE^N)$ gives
\begin{align*}
  \cH\big(\boldsymbol\Psi(\mu)|\gamma_N\big)
  =
 \int \cH\big((\Psi_{m,E/d})_\#\mu|(\Psi_{m,E/d})_\#\pi_N\big)\dd Q_N(m,E)
 = \cH(\mu|\pi_N)\;,
\end{align*}
where the last equality follows from \eqref{eq:ent-push} and
$\Psi_{m,E}$ being bijective. Since $(L_N)_\#\mu^N\rightharpoonup c$ implies
$(L_N)_\#\boldsymbol\Psi(\mu^N)\rightharpoonup c$, we can now 
deduce \eqref{eq:entropy-conv} from
\eqref{eq:entropy-conv-gamma}.
\end{proof}

\appendix
\section{The collision distance}\label{sec:metric}

In this section, we present a new type of distance between probability
measures on $\R^d$ which is formally the Riemannian distance
associated to the Onsager operator $\mathcal K^B$, see \eqref{eq:Onsager}.  The Riemannian
distance $\cW_B$ between to probability densities $f_0,f_1$ is
formally given as
\begin{align}\label{eq:intro-def-W}
  \cW_B(f_0,f_1)^2 = \inf \left\{\frac14\int_0^1\int |\dgrad\psi_t|^2\Lambda(f_t)B(v-\vs,\omega)\dd\omega\dd\vs\dd
    v\dd t \right\}\;,
\end{align}
where the infimum runs over all curves of densities $t\mapsto f_t$ connecting $f_0$ to $f_1$ and all functions
$\psi:[0,1]\times\R^d\to\R$ related via
\begin{align}\label{eq:intro-ce}
  \partial_tf_t(v)+\int\dgrad\psi_t\Lambda(f_t)B(v-\vs,\omega)\dd\omega\dd\vs
  = 0\;.
\end{align}
Note that the definition of $\cW_B$ resembles the dynamic formulation
of the $L^2$-Wasserstein distance, known as the Benamou--Brenier
formula \cite{BB00}. Here, the \emph{collision rate equation}
\eqref{eq:intro-ce} takes over the role of the usual continuity
equation.

The distance $\cW_B$ will be constructed by relaxing the minimization
problem above to a measure valued framework and by minimizing the
action as defined in Section \ref{sec:CRE-action} over curves
connecting two given probability measures via the collision rate
equation.

In this section, we will relax the assumptions on the collision kernel
and require:
\begin{assumption}\label{ass:standing}
  $B:\R^d\times S^{d-1}\to\R_+$ is measurable, invariant under the
  transformation \eqref{eq:pre-post2}, such that
  $k\mapsto B(k,\omega)$ is continuous for a.e.~$\omega$ and there
  exist constants $\gamma\in(-\infty,1]$ and $c_B>0$ such that
  \begin{align}\label{eq:ass-standing}
     \int_{S^{d-1}}B(k,\omega)\dd\omega \leq c_B\ip{k}^\gamma \quad\forall k\in\R^d\;.
  \end{align} 
\end{assumption}

The following result will allow us to extract subsequential limits
from sequences of solutions to the collision rate equation with
uniform action and moment bounds. 

Given $p\geq 1$ and $E>0$ we will write
\[\CE_T^{p,E}:=\big\{(\mu,\cU)\in\CE_T~:~\mu_t\in\cPpE~\forall t\in[0,T]\big\}\;.\]

\begin{proposition}[Compactness of solutions with bounded action and
  moments]\label{prop:cecompactness}
  Let $(\mu^n,\cU^n)$ be a sequence in $\CE_{T}^{p,E}$ with $p\geq 2$ such that 
\begin{equation}\label{eq:bdd-energy}
 \sup\limits_n\int_0^T\cA(\mu_t^n,\cU_t^n) \dd t~<~\infty\;.
\end{equation}
Then there exists a couple
$(\mu,\cU)\in\CE_{T}^{p,E}$ such that up to
extraction of a subsequence
\begin{align*}
 & \mu^n_t\rightharpoonup\mu_t\quad \mbox{weakly in }\cP(\R^d)\ \mbox{for all } t\in[0,T]\ ,\\
 & \cU^n\rightharpoonup^*\cU\quad \mbox{weakly* in }\cM(\Omega\times[0,T])\ .
\end{align*}
Moreover, along this subsequence we have :
\begin{equation*}
 \int_0^T\cA(\mu_t,\cU_t) \dd t~\leq~\liminf\limits_n \int_0^T\cA(\mu_t^n,\cU_t^n) \dd t\ .
\end{equation*}
\end{proposition}

\begin{proof}
  Thanks to the uniform bounds on action and moments, we can proceed
  verbatim as in the proof of Proposition
  \ref{prop:opt-collision-rate} to obtain existence of a Borel family
  $(\cU_t)_{t\in[0,T]}$ satisfying (iii) of Definition \ref{def:ce}
  such that $\cU^n_t\dd t$ converges weakly* to $\cU_t\dd t$ and the
  convergence \eqref{eq:converge-bnu2} holds.  By a further argument
  based on \eqref{eq:unif-integrability3}, we can approximate the
  indicator function $\one_{(t_0,t_1)}$ for any $0\leq t_0<t_1\leq T$
  by functions $a\in C\big([0,T]\big)$ and obtain for any
  $\xi\in C_b(\R^d)$:
  \begin{align}\label{eq:converge-bnu}
    \int_{t_0}^{t_1}\int\dgrad\xi \dd\cU^n_t
    \dd t~\overset{n\to\infty}{\longrightarrow}~\int_{t_0}^{t_1}\int\dgrad\xi
    \dd\cU_t \dd t\ .
  \end{align}
  Finally, we show existence of a limiting curve
  $(\mu_t)_{t\in[0,T]}$. Since $\cPpE$ is compact w.r.t.~weak
  convergence, after extraction of another subsequence we can assume
  that $\mu^n_0\rightharpoonup\mu_0$ weakly for some
  $\mu_0\in\cP(\R^d)$. Using this, the convergence
  \eqref{eq:converge-bnu} and the collision rate equation in the form
  \eqref{eq:cedistributionrefined} infer that $\mu^n_t$ converges
  weakly to some probability measure $\mu_t$ for every $t\in[0,T]$ and
  that $(\mu,\cU)$ satisfies
  \eqref{eq:cedistributionrefined}. In particular, $t\mapsto\mu_t$ is
  weakly continuous and hence $(\mu,\cU)\in\CE_T$. By lower
  semicontinuity of moments, we infer $\cE_p(\mu_t)\leq E$ for all
  $t$. The lower semicontinuity statement follows from Lemma
  \ref{lem:Fprops} by noting that
  $\int_0^T\cA(\mu^n_t,\cU^n_t)\dd
  t=\cF_\alpha(\mu^{n,1},\mu^{n,2},\cU^n)$
  with $\mu^{n,k}=\mu^{n,k}_t\dd t$.
\end{proof}

Given $p\geq 2$ and $E>0$ now define the following distance

\begin{definition}[Distance]\label{def:metric}
 For $\mu_0,\mu_1\in\cPpE$ we define
\begin{equation}\label{eq:defmetric}
\cW_B(\mu_0,\mu_1)^2:=\inf\left\{\int_0^1\cA(\mu_t,\cU_t)\dd t : (\mu,\cU)\in\CE_{1}^{p,E}(\mu_0,\mu_1)\right\}\;,
\end{equation}
with the convention that $\cW_B(\mu_0,\mu_1)=+\infty$ if
the set over which the infimum is taken, is empty.
\end{definition}

\begin{remark}\label{rem:moments}
  In the same way one could construct an (a priori smaller) extended
  distance on the full space $\cP(\R^d)$ by dropping the moment constraint
  and minimize over $\CE_1$ instead of $\CE_1^{p,E}$. We will not explore this here. We
  stress that $\cW_{B}$ defined as above depends implicitly on the choice of $p$ and $E$.
\end{remark}

Let us give an equivalent characterization of the infimum in
\eqref{eq:defmetric}.

\begin{lemma}\label{lem:equivcharacter}
For any $T>0$ and $\mu_0,\mu_1\in\cPpE$ we have :
\begin{align*}
\cW_B(\mu_0,\mu_1)~=~\inf\left\{\int_0^T
\sqrt{\cA(\mu_t,\cU_t)}\dd t\ :\quad
(\mu,\cU)\in\CE_{T}^{p,E}(\mu_0,\mu_1)\right\}\ .
\end{align*}
\end{lemma}

\begin{proof}
  This follows from a standard reparametrization argument. See
  \cite[Lem. 1.1.4]{AGS08} or \cite[Thm. 5.4]{DNS09} for details in
  similar situations.
\end{proof}

The next result shows that the infimum in the definition above is
in fact a minimum.

\begin{proposition}\label{prop:minimizers}
  Let $\mu_0,\mu_1\in\cPpE$ be such that
  $W:=\cW_B(\mu_0,\mu_1)$ is finite. Then the infimum in
  \eqref{eq:defmetric} is attained by a curve
  $(\mu,\cU)\in\CE_{1}^E(\mu_0,\mu_1)$ satisfying
  $\cA(\mu_t,\cU_t)=W^2$ for a.e. $t\in[0,1]$.
\end{proposition}

\begin{proof}
  Existence of a minimizing curve
  $(\mu,\cU)\in\CE_{1}^E(\mu_0,\mu_1)$ follows immediately by
  the direct method taking into account Proposition
  \ref{prop:cecompactness}. Invoking Lemma \ref{lem:equivcharacter}
  and Jensen's inequality we see that this curve satisfies
  \begin{align*}
    \int_0^1\sqrt{\cA(\mu_t,\cU_t)}\dd
    t~\geq~W~=~\left(\int_0^1\cA(\mu_t,\cU_t)\dd
      t\right)^{\frac{1}{2}}~\geq~\int_0^1\sqrt{\cA(\mu_t,\cU_t)}\dd t\ .
  \end{align*}
  Hence we must have $\cA(\mu_t,\cU_t)=W^2$ for a.e. $t\in[0,1]$.
\end{proof}
 
We have the following properties of the function $\cW_B$.

\begin{theorem}\label{thm:distance}
  $\cW_B$ defines an (extended) distance on $\cPpE$. The topology it
  induces is stronger than the weak topology and bounded sets
  w.r.t.~$\cW_B$ are weakly compact. Moreover, the map
  $(\mu_0,\mu_1)\mapsto \cW_B(\mu_0,\mu_1)$ is lower semicontinuous
  w.r.t.~weak convergence. For each $\tau\in\cPpE$ the set
  $\cP_\tau:=\{\mu\in\cPpE\ :\ \cW_B(\mu,\tau)<\infty\}$ equipped
  with the distance $\cW_B$ is a complete geodesic space.
\end{theorem}

Here, we call a function $d:X\times X\to[0,\infty]$ an \emph{extended
  distance} on the set $X$, if it is symmetric, satisfies the triangle
inequality and vanishes precisely on the diagonal.

\begin{proof}
  Symmetry of $\cW_B$ is obvious from the fact that
  $\a(w,\cdot,\cdot)=\a(-w,\cdot,\cdot)$. Equation
  \eqref{eq:cedistributionrefined} shows that two curves in
  $\CE_{1}^{p,E}$ can be concatenated to obtain a curve in
  $\CE_{2}^{p,E}$. Hence the triangle inequality follows easily using
  Lemma \ref{lem:equivcharacter}. To see that $\cW_B(\mu_0,\mu_1)>0$
  whenever $\mu_0\neq\mu_1$ assume that $\cW_B(\mu_0,\mu_1)=0$ and
  choose a minimizing curve $(\mu,\cU)\in\CE_{1}^{p,E}(\mu_0,\mu_1)$. Then
  we must have $\cA(\mu_t,\cU_t)=0$ and hence $\cU_t=0$ for
  a.e. $t\in(0,1)$. From the continuity equation in the form
  \eqref{eq:cedistributionrefined} we infer $\mu_0=\mu_1$.
 
  The compactness assertion and lower semicontinuity of $\cW_B$ follow
  immediately from Proposition \ref{prop:cecompactness}. These in turn
  imply that the topology induced by $\cW_B$ is stronger than the weak
  one.

  Let us now fix $\tau\in\cPpE$ and let $\mu_0,\mu_1\in
  \cP_\tau$. By the triangle inequality we have
  $\cW_B(\mu_0,\mu_1)<\infty$ and hence Proposition
  \ref{prop:minimizers} yields existence of a minimizing curve
  $(\mu,\cU)\in\CE_{1}^{p,E}(\mu_0,\mu_1)$. The curve
  $t\mapsto\mu_t$ is then a constant speed geodesic in $\cP_\tau$
  since it satisfies
  \begin{align*}
    \cW_B(\mu_s,\mu_t)~=~\int\limits_s^
    t\sqrt{\cA(\mu_r,\cU_r)}\dd r~=~(t-s)\cW_B(\mu_0,\mu_1)\quad\forall
    0\leq s\leq t\leq 1\ .
  \end{align*}
  To show completeness, let $(\mu^n)_n$ be a Cauchy sequence in
  $\cP_\tau$. In particular the sequence is bounded w.r.t.~$\cW_B$ and
  we can find a subsequence (still indexed by $n$) and
  $\mu^\infty\in\cPpE$ such that $\mu^n\rightharpoonup\mu^\infty$
  weakly. Invoking lower semicontinuity of $\cW_B$ and the Cauchy
  condition we infer that $\cW_B(\mu^n,\mu^\infty)\to 0$ as $n\to\infty$
  and that $\mu^\infty\in \cP_\tau$.
\end{proof}

It is not yet clear when precisely the distance $\cW_B$ is
finite. However, it is easily seen to be finite along solutions to
the Boltzmann equation: if $f_t$ is a solution according to Theorem
\ref{thm:boltzmann} and we set $\mu_t=f_t\cL$ and
\begin{align*}
   \cU_t=\dgrad\log f_t\Lambda(f_t)\cB = \big[(\fp)_t(\fsp)_t-f_t(\fs)_t\big]\cB\;,
 \end{align*}
then $(\mu,\cU)\in\CE^E$ and we have $\cA(\mu_t,\cU_t)=D(\mu_t)$. Thus,
\begin{align*}
  \cW_B(\mu_0,\mu_T) \leq \int_0^T\sqrt{D(\mu_t)}\dd t\leq \sqrt{T} \left(\int_0^TD(\mu_t)\dd t\right)^{\frac12} = \sqrt{T}\sqrt{\cH(\mu_0)-\cH(\mu_T)}\;.
\end{align*}

The following result shows that the distance $\cW_B$ can be bounded
from below by the $L^1$-Wasserstein distance. Recall that the
$L^1$-Wasserstein distance is defined for $\mu_0,\mu_1\in\cP(\R^d)$ by
\begin{align*}
  W_{1}(\mu_0,\mu_1)~:=~\inf\limits_\pi \int |x-y|\pi(\dd x,\dd y)\ ,
\end{align*}
where the infimum is taken over all probability measures
$\pi\in\cP(\R^d\times\R^d)$ whose first and second marginal are $\mu_0$
and $\mu_1$ respectively.

\begin{proposition}\label{prop:lowerboundW} 
  Let $p\geq 2+\max(\gamma,0)$. For any $\mu_0,\mu_1\in\cPpE$ we have the bound
  \begin{align*}
    W_{1}(\mu_0,\mu_1)~\leq~\sqrt{2c_BE}\cW_B(\mu_0,\mu_1)\ .
  \end{align*}
\end{proposition}

\begin{proof}
  We can assume that $\cW_B(\mu_0,\mu_1)<\infty$. Take a minimizing
  curve $(\mu,\cU)\in\CE^{p,E}_1(\mu_0,\mu_1)$ and let $\phi:\R^d\to\R$ be
  a bounded $1$-Lipschitz function. This implies that $|\dgrad
  \phi|\leq 2|v-\vs|$. Taking into account Remark \ref{rem:generaltest}
  and using Lemma \ref{lem:integrability}, we estimate
  \begin{align*}
    &\abs{\int\varphi \dd\mu_1-\int\varphi \dd\mu_0}~=~\frac14\abs{\int_0^1\int\dgrad\varphi \dd\cU_t\dd t}\\
    &\leq~\frac12\int_0^1\int |v-\vs|\dd\abs{\cU_t}(v,\vs,\omega)\dd t\\
    &\leq~\left(\int_0^1\cA(\mu_t,\cU_t)\dd t\right)^\frac12\left(\int_0^1\int |v|^2+|\vs|^2B(v-\vs,\omega)\mu_t(\dd v)\mu_t(\dd \vs)\dd t\right)^\frac12\\
    &\leq~\sqrt{2c_BE}\cW_B(\mu_0,\mu_1)\ .
  \end{align*}
  Here we have also used \eqref{eq:ass-standing} and the fact that
  $\mu_t$ has $p$-moment less than $E$ in the last inequality. Taking the
  supremum over all bounded $1$-Lipschitz functions $\phi$ yields the
  claim by Kantorovich--Rubinstein duality (see \cite[Thm. 5.10,
  5.16]{Vil09}).
\end{proof}

We now give a characterization of absolutely continuous curves with
respect to $\cW_B$. See \eqref{eq:abs-continuous} and
\eqref{eq:def-md} for the definition of absolutely continuous curves
and their metric derivative.

\begin{proposition}[Metric velocity]\label{prop:metricderivative}
  A curve $(\mu_t)_{t\in[0,T]}$ in $\cPpE$ is absolutely
  continuous with respect to $\cW_B$ if and only if there exists a
  Borel family $(\cU_t)_{t\in[0,T]}$ such that $(\mu,\cU)\in\CE^{p,E}_{T}$
  and
  \begin{align*}
    \int_0^T\sqrt{\cA(\mu_t,\cU_t)}\dd t~<~\infty\ .
  \end{align*}
  In this case, the metric derivative is bounded as $\abs{\dot \mu}^2(t)\leq\cA(\mu_t,\cU_t)$ for
  a.e. $t\in[0,T]$. Moreover, there exists a unique Borel family
  $\tilde{\cU}_t$ with $(\mu,\tilde\cU)\in\CE_{T}^{p,E}$ such that
  \begin{align}\label{eq:optimalvelocity}
    \abs{\dot \mu}^2(t)=\cA(\mu_t,\tilde{\cU}_t)\qquad \text{for a.e. }t\in[0,T]\ .
  \end{align}
\end{proposition}

\begin{proof}
  The proof follows from the very same arguments as in
  \cite[Thm. 5.17]{DNS09}.
\end{proof}

We can describe the optimal velocity measures $\tilde\cU_t$ appearing
in the preceding proposition in more detail.  We define
$T_\mu$ to be the set of all $\cU\in\cM(\Omega)$ such that
$\cA(\mu,\cU)<\infty$ and $\cA(\mu,\cU)~\leq~\cA(\mu,\cU+\bbeta)$ for
all $\bbeta\in \cM(\Omega)$ satisfying
\begin{align*}
\frac14\int_\Omega\dgrad\xi\dd\bbeta~=~0\qquad\forall\xi\in C^\infty_c(\Omega)\ .
\end{align*}

\begin{corollary}\label{cor:tangentspace}
  Let $(\mu,\cU)\in\CE_{T}^{p,E}$ such that the curve $t\mapsto\mu_t$ is
  absolutely continuous w.r.t. $\cW_B$. Then $\cU$ satisfies
  \eqref{eq:optimalvelocity} if and only if $\cU_t\in
  T_{\mu_t}$ for a.e. $t\in[0,T]$.
\end{corollary}

If $\mu$ is absolutely continuous with respect to Lebesgue measure
$\cL$ we can give an explicit description of $T_\mu$. Recall that
$\cB\in\cM(\Omega)$ is the measure given by
$\dd\cB(v,\vs,\omega)=B(v-\vs,\omega)\dd v\dd\vs\dd\omega$.

\begin{proposition}\label{prop:tangentspace-density} 
  Let $\mu=f m\in\cPpE$. Then we have $\cU\in T_\mu$
  if and only if $\cU=U\Lambda(f) \cB$ is absolutely continuous
  w.r.t.~the measure $\Lambda(f)\cB$ and
  \begin{align*}
    U~\in~ \overline{\{\dgrad\phi\ \vert\ \phi\in C^\infty_c(\R^d)\}}^{L^2(\Lambda(f)\cB)}~=:~T_f\ .
  \end{align*}
\end{proposition}

\begin{proof}
  If $\cA(\mu,\cU)$ is finite we infer from Lemma \ref{lem:densities}
  that $\cU=U\Lambda(f)\cB$ for some density $U:\Omega\to\R$ and that
  $\cA(\mu,\cU)=\norm{U}^2_{L^2(\Lambda(f)\cB)}$. Now the
  optimality condition in the definition of $T_\mu$ is equivalent
  to
  \begin{align*}
    \norm{U}_{L^2(\Lambda(f)\cB)}~\leq~\norm{U+V}_{L^2(\Lambda(f)\cB)}\qquad \forall V\in N_f\ ,
  \end{align*}
  where $N_f:=\{V\in L^2(\Lambda(f)\cB)\;:\;\int\dgrad\xi
  V\Lambda(f)\cB = 0\ \forall \xi\in C^\infty_c(\R^d)\}$. This
  implies the assertion of the proposition after noting that $N_f$
  is the orthogonal complement in $L^2$ of $T_f$.
\end{proof}

In the light of the formal Riemannian interpretation of the distance
$\cW_B$ one should view $T_\mu$ as the tangent space at the measure
$\mu$. This is reminiscent of Otto's Riemannian interpretation of the
$L^2$-Wasserstein space \cite{O01}.

\section{Metric gradient flow}
\label{sec:metricgf}
In this section, we recast the variational characterization of Section
\ref{sec:gradflow} in the language of the theory of gradient flows in
metric spaces. Let us briefly recall the basic theory of gradient flow
in metric spaces. For a detailed account we refer the reader to
\cite{AGS08}.

Let $(X,d)$ be a complete metric space and let
$E:X\to(-\infty,\infty]$ be a function with proper domain, i.e.~the
set $D(E):=\{x:E(x)<\infty\}$ is non-empty.

A curve $(x_t)_{t\in (a,b)}$ in $(X,d)$ is called $p$-absolutely continuous for
$p\geq1$ if there exists $m\in L^p((a,b))$ such that
\begin{align}\label{eq:abs-continuous}
  d(x_s,x_t)~\leq~\int_s^tm(r)\dd r \quad\forall~a\leq s\leq t\leq
  b\ .
\end{align}
In this case we write $x\in AC^p\big((a,b);(X,d)\big)$. For $p=1$ we
simply drop $p$ in the notation. Similarly, one defines locally
$p$-absolutely continuous curves. For a locally absolutely continuous
curve the metric derivative defined by
\begin{align}\label{eq:def-md}
  \abs{\dot x}(t)~:=~\lim\limits_{h\to0}\frac{d(x_{t+h},x_t)}{\abs{h}}
\end{align}
exists for a.e.~$t$ and is the minimal $m$ in
\eqref{eq:abs-continuous}, see \cite[Thm.1.1.2]{AGS08}.

The following notion plays the role of the modulus of the gradient in
a metric setting.

\begin{definition}[Strong upper gradient]\label{def:upper-grad}
  A function $g:X\to[0,\infty]$ is called a \emph{strong upper
    gradient} of $E$ if for any $x\in AC\big((a,b);(X,d)\big)$ the
  function $g\circ x$ is Borel and
  \begin{align*}
 |E(x_s)-E(x_t)|~\leq~\int_s^tg(x_r)|\dot x|(r)\dd r \quad\forall~a\leq s\leq t\leq
  b\ .
 \end{align*}
\end{definition}

Note that by the definition of strong upper gradient, and Young's
inequality $ab\leq \frac12(a^2+b^2)$, we have that for all $s\leq t$:
\begin{align*}
    E(x_t) - E(x_s) +\frac12 \int_s^t g(x_r)^2 + |\dot x|^2(r)\dd r\geq 0\;.
\end{align*}

\begin{definition}[Curve of maximal slope]\label{def:curve-max-slope}
  A locally $2$-absolutely continuous curve $(x_t)_{t\in(0,\infty)}$ is
  called a curve of maximal slope of $E$ w.r.t.~its strong upper
  gradient $g$ if $t\mapsto E(x_t)$ is non-increasing and
  \begin{align}\label{eq:cms}
    E(x_t) - E(x_s) +\frac12 \int_s^t g(x_r)^2 + |\dot x|^2(r)\dd r \leq 0 \quad\forall~0< s\leq t\;.
  \end{align}
  We say that a curve of maximal slope starts from $x_0\in X$ if $\lim_{t\searrow 0}x_t=x_0$.
\end{definition}

Equivalently, we can require equality in \eqref{eq:cms}. If a strong
upper gradient $g$ of $E$ is fixed we also call a curve of maximal
slope of $E$ (relative to $g$) a \emph{gradient flow curve}.

Finally, we define the (descending) metric slope of $E$ as the
function $|\partial E|:D(E)\to[0,\infty]$ given by
\begin{align}\label{eq:metric-slope-def}
  |\partial E| (x) = \limsup_{y\to x}\frac{\max\{E(x)-E(y),0\}}{d(x,y)}\;.
\end{align}

The metric slope is in general only a weak upper gradient $E$, see
\cite[Thm.~1.2.5]{AGS08}.  In our application to the homogeneous
Boltzmann equation, we will show that the square root of the
dissipation $D$ provides a strong upper gradient for the entropy
$\cH$.

Let us assume that $p\geq 2+\max(\gamma,0)$ and that the collision
kernel $B$ satisfies Assumption \ref{ass:kernel-bounded}. Then we have
the following

\begin{corollary}[Boltzmann equation as curve of maximal slope]\label{cor:sug-dissipation}
  $\sqrt{D}$ is a strong upper gradient for $\cH$ on
  $(\cPpE,\cW_B)$. Moreover, for any $\mu_0\in\cPpE$ with
  $\cH(\mu_0)<\infty$, the curves of maximal slope of $\cH$ w.r.t.~the
  strong upper gradient $\sqrt{D}$ starting from $\mu_0$ are precisely
  the solutions to the Boltzmann equation satisfying
  \eqref{eq:dissipation-integrable}.
\end{corollary}

\begin{proof}
  Let $(\mu_r)_r$ be an absolutely continuous curve such that
  $\int_s^t\sqrt{D(\mu_r)}|\dot \mu|(r)\dd r<\infty$. This implies
  that $\mu_r$ has a density $f_r$ (and hence by Lemma
  \ref{lem:densities} $\cU_r$ has a density $U_r$) for a.e.~$r$. We
  can also assume that one of the measures $\mu_s,\mu_t$ has finite
  entropy, say $\mu_s$. Then, Proposition \ref{prop:chainrule}
  together with the estimate \eqref{eq:bound-DA} yield immediately
  that $\sqrt{D}$ is a strong upper gradient. Theorem \ref{thm:EDI}
  gives the identification of curves of maximal slope.
\end{proof}

\section{Variational approximation scheme}
\label{sec:minmov}

In this section, we consider a time-discrete variational approximation
scheme  for the homogeneous Boltzmann equation. Recall that we make
Assumption \ref{ass:kernel-bounded} on the collision kernel $B$ and let $p\geq 2+\max(\gamma,0)$. The
scheme can be interpreted as the implicit Euler scheme for the
gradient flow equation. Given a time step $\tau>0$ and an initial
datum $\mu_0\in \cPpE$ with $\cH(\mu_0)<\infty$, we consider a
sequence $(\mu^\tau_n)_n$ in $\cPpE$ defined recursively via

\begin{align}\label{eq:min-move}
  \mu^\tau_0 = \mu_0\;,\quad \mu^\tau_n \in \underset{\nu}{\argmin} \Big[\cH(\nu)
  + \frac{1}{2\tau}\cW_B(\nu,\mu^\tau_{n-1})^2\Big]\;.
\end{align}

Then we build a discrete gradient flow trajectory as the piece-wise
constant interpolation $(\bar\mu^\tau_t)_{t\geq 0}$ given by
\begin{align}\label{eq:interpolation}
  \bar\mu^\tau_0 = \mu_0\;, \quad \bar\mu^\tau_t = \mu^\tau_n\ \text{if } t\in\big((n-1)\tau,n\tau]\;. 
\end{align}

Then we have the following result.

\begin{theorem}\label{thm:JKO}
  For any $\tau>0$ and $\mu_0\in \cPpE$ with $\cH(\mu_0)<\infty$
  the variational scheme \eqref{eq:min-move} admits a solution
  $(\mu^\tau_n)_n$. As $\tau\to0$, for any family of discrete
  solutions there exists a sequence $\tau_k\to0$ and a locally
  $2$-absolutely continuous curve $(\mu_t)_{t\geq0}$ such that
  \begin{align}\label{eq:scheme-limit}
    \bar\mu^{\tau_k}_t \rightharpoonup \mu_t\quad \forall
    t\in[0,\infty)\;.
  \end{align}
  Moreover, any such limit curve is a gradient flow of the entropy,
  i.e.~ a solution to the Boltzmann equation satisfying
  \eqref{eq:dissipation-integrable}.
\end{theorem}

With the knowledge that the Boltzmann equation in our setting has a
unique solution (assuming in addition $\cE_4(\mu_0)<\infty)$ if
$\gamma>0$), we obtain convergence of $\bar\mu^\tau_t$ to the solution
to the Boltzmann equation for any sequence of time steps $\tau\to0$.

With the work we have done so far, Theorem \ref{thm:JKO} follows
basically from standard general results for metric gradient flows
where \eqref{eq:min-move} is known as the minimizing movement scheme,
see \cite[Sec.~2.3]{AGS08}. We need one small additional ingredient
relating the dissipation $D$ to the metric slope $|\partial \cH|$ of
the entropy in the metric space $(\cPpE,\cW_B)$. Recall
\eqref{eq:metric-slope-def} for the definition of the metric slope. We
consider its sequentially lower semicontinuous envelope, or
\emph{relaxed slope} $|\partial^-\cH|$ given by
\begin{align*}
  |\partial^-\cH|(\mu)=\inf\left\{
\liminf_{n\to\infty}|\partial\cH(\mu_n)| ~:~\mu_n\rightharpoonup \mu,~ \sup_n\{\cW_B(\mu_n,\mu),\cH(\mu_n)\}<\infty
 \right\}\;.
\end{align*}

\begin{lemma}\label{lem:slope-diss}
  For any $\mu\in\cPpE$ with $\cH(\mu)<\infty$ we have that
   $ \sqrt{D(\mu)} \leq |\partial^- \cH(\mu)|\;.$ 
In particular, $|\partial^- \cH(\mu)|$ is a strong upper gradient for $\cH$.
\end{lemma}

\begin{proof}
  Let $f$ be the density of $\mu$ and consider the solution $(f_t)$
  to the homogeneous Boltzmann equation with initial datum
  $f$. Set $\mu_t=f_t\cL$ and observe that
  \begin{align*}
    D(f)&\leq\lim_{t\searrow0} \frac{\cH(\mu)-\cH(\mu_t)}{t} = \lim_{t\searrow0} \frac{\cH(\mu)-\cH(\mu_t)}{\cW_B(\mu_t,\mu)}\frac{\cW_B(\mu_t,\mu)}{t}\\
        &\leq |\partial\cH(\mu)||\dot\mu|(0) \leq |\partial\cH(\mu)|\sqrt{D(\mu)}\;.
  \end{align*}
  Thus, we have $\sqrt{D(\mu)}\leq|\partial\cH(\mu)|$ for any such
  $\mu$. The claim follows immediately from the lower semicontinuity
  of $D$, Lemma \ref{lem:diss-lsc}.
\end{proof}

\begin{proof}[Proof of Theorem \ref{thm:JKO}]
  We verify that the present situation is consistent with the
  abstract setting considered in \cite[Sec.~2]{AGS08}.

  We consider the metric space $(\cP_{\mu_0},\cW_B)$ and endow
  it with the weak topology $\sigma$. By Theorem \ref{thm:distance},
  $(\cP_{\mu_0},\cW_B)$ is complete, $\cW_B$ is lower
  semicontinuous w.r.t.~$\sigma$ and induces a stronger topology.
  Recall from Section \ref{sec:prelim} that the entropy $\cH$ is
  bounded below on $\cPpE$ and lower semicontinuous w.r.t.~weak
  convergence. Moreover, $\cPpE$ is compact w.r.t.~weak
  convergence. Thus, \cite[Assumption 2.1 a,b,c]{AGS08} are satisfied.
 
  Existence of a solution to the variational scheme
  \eqref{eq:min-move} and of a subsequential limit curve $(\mu_t)_t$
  now follows from \cite[Cor.~2.2.2, Prop.~2.2.3]{AGS08}. Moreover,
  \cite[Thm.~2.3.2]{AGS08} gives that the limit curve is a curve of
maximal slope for the strong upper gradient $|\partial^-\cH|$, i.e.~
  \begin{align*}
     \frac12\int_0^t|\dot\mu|^2(r) + |\partial^-\cH(\mu_r)|^2\dd r + \cH(\mu_t)
      \leq \cH(\mu_0)\;.
   \end{align*}
   Thus, by Lemma \ref{lem:slope-diss}, it is also a curve of maximal
   slope for the strong upper gradient $\sqrt{D}$. 
 \end{proof} 

\bibliographystyle{plain}
\bibliography{literature}

\end{document}